\documentclass[10pt]{amsart}
\usepackage{amssymb, amsmath}
\usepackage{graphics}
\usepackage{latexsym}
\usepackage{amsmath}
\usepackage{amssymb,amsthm,amsfonts}
\usepackage{amscd}
\usepackage[arrow, matrix, curve]{xy}
\usepackage{syntonly}
\usepackage{yfonts}
\usepackage{tikz-cd}
\usepackage{filecontents}
\usepackage{mathtools}
\usepackage{stmaryrd}
\usepackage{MnSymbol}
\usepackage{amscd}
\usepackage{mathrsfs} 
\usepackage[arrow, matrix, curve]{xy}
\usepackage{booktabs}
\usepackage{float}
\floatstyle{plaintop}
\restylefloat{table}

\DeclareMathOperator{\ord}{ord}
\DeclareMathOperator{\gl}{GL}

\DeclareMathOperator{\gsp}{Gsp}
\DeclareMathOperator{\br}{Br}
\DeclareMathOperator{\ram}{Ram}
\DeclareMathOperator{\jac}{Jac}
\DeclareMathOperator{\irr}{Irr}
\DeclareMathOperator{\sym}{Sym}
\DeclareMathOperator{\nm}{Nm}
\DeclareMathOperator{\im}{Im}
\DeclareMathOperator{\pic}{Pic}
\DeclareMathOperator{\deck}{Deck}
\theoremstyle{plain}
\newtheorem{thm}{Theorem}[section]
\newtheorem{theorem}[thm]{Theorem}
\newtheorem{lemma}[thm]{Lemma}

\theoremstyle{definition}

\newtheorem{remark}[thm]{Remark}

\newtheorem{definition}[thm]{Definition}

\numberwithin{equation}{thm}

\newcommand{\sC}{{\mathcal C}}

\newcommand{\C}{{\mathbb C}}
\newcommand{\D}{{\mathbb D}}

\renewcommand{\H}{{\mathbb H}}

\renewcommand{\L}{{\mathbb L}}

\renewcommand{\P}{{\mathbb P}}
\newcommand{\Q}{{\mathbb Q}}
\newcommand{\R}{{\mathbb R}}

\newcommand{\Z}{{\mathbb Z}}

\newcommand{\Hom}{{\rm Hom}}

\newcommand{\Sp}{{\rm Sp}}

\begin{document}
	\title{Shimura subvarieties in the Prym locus of ramified Galois coverings}
	\author[G.P. Grosselli]{Gian Paolo Grosselli} 
	\address{Dipartimento di Matematica, Universit\`a di Pavia, via Ferrata 5, I-27100, Pavia, Italy }
	\email{g.grosselli@campus.unimib.it}
	\author[A. Mohajer]{Abolfazl Mohajer}
	\address{Universit\"{a}t Mainz, Fachbereich 08, Institut f\"ur Mathematik, 55099 Mainz, Germany}
	\email{mohajer@uni-mainz.de} 
	\thanks{The first author is member of GNSAGA of INdAM.
	The first author was partially supported by national MIUR funds,
	PRIN  2017 Moduli and Lie theory and by MIUR: Dipartimenti di Eccellenza Program (2018-2022) - Dept. of Math. Univ. of Pavia.}
	\maketitle
	
	\begin{abstract}
		We study Shimura (special) subvarieties in the moduli space $A_{p,D}$ of complex abelian varieties of dimension $p$ and polarization type $D$. These subvarieties arise from families of covers compatible with a fixed group action on the base curve such that the quotient of the base curve by the group is isomorphic to $\P^1$. We give a criterion for the image of these families under the Prym map to be a special subvariety and, using computer algebra, obtain 210 Shimura subvarieties contained in the Prym locus. 
	\end{abstract}
	
	\section{Introduction}
	In \cite{CFGP}, E. Colomobo, P. Frediani, A. Ghigi and M. Penegini have extensively studied  Shimura curves of PEL type in $A_g$,  contained generically in the Prym locus (see also the papers \cite{CF} and \cite{CF2} by E. Colomobo and P. Frediani). The general set-up is as follows: Let $R_g$ be the scheme of isomorphism classes $[C,\eta]$, for $C$ a smooth projective curve of genus $g$ and $\eta\in \pic^0(C)$ a 2-torsion element, i.e., $\eta\neq \mathcal{O}_C$ but $\eta^2=\mathcal{O}_C$. The line bundle $\eta$ determines an (unramified) \'etale double cover $h:\tilde{C}\to C$ and there is an induced norm map $\nm:\pic^0(\tilde{C})\to \pic^0(C)$. The Prym variety associated to $[C,\eta]$ is defined to be the connected component of $\ker\nm$ containing the identity and is denoted by $P(C,\eta)$ or $P(\tilde{C},C)$. In a similar way, one can construct a Prym variety for ramified covers. Consider the scheme $R_{g,2}$ parametrizing pairs $[C,B,\eta]$ up to isomorphism, where $C$ is a smooth projective curve of genus $g$, $\eta$ a line bundle on $C$ of degree $1$, and $B$ a divisor in the linear series $|\eta^2|$ corresponding to a double covering $\pi:\tilde{C}\to C$ ramified over $B$. The assignment $[C,\eta]\mapsto P(\tilde{C}, C)$ (resp. $[C,B,\eta]\mapsto P(\tilde{C}, C)$) defines a map $\mathscr{P}:R_{g}\to A_g$ (resp. $\mathscr{P}:R_{g,2}\to A_g$). This goes under the name of the \emph{Prym map}. Both in the unramified Prym locus corresponding to the unramified \'etale double covers and the ramified Prym locus, corresponding to the family of double covers ramified at two points, \cite{CFGP} gives examples of one-parameter families $(C_t,\eta_t)$ ($t\in T=\P^1\setminus\{0,1,\infty\}$) for which the image $\mathscr{P}(T)$ parametrizes Shimura curves in $A_g$ provided that the curves $C_t$ are themselves Galois coverings of $\mathbb{P}^{1}$ (here we denote the image of  $T$ in $ R_g$ via $t\mapsto [C_t,\eta_t]$ again by $T$ see section 2 or \cite{CFGP}). More precisely, the authors consider a family of Galois covers $\tilde{C_t}\to \mathbb{P}^{1}$ with Galois group $\tilde{G}$ and a central involution $\sigma$ such that the double covering $\tilde{C_t}\to \tilde{C_t}/\langle \sigma\rangle$ is either \'etale or ramified over exactly two distinct points. By the theory of coverings, the Galois covering $\tilde{C_t}\to \mathbb{P}^{1}$ is determined by an epimorphism $\tilde{\theta}:\Gamma_r\to\tilde{G}$ with branch points $t_1,\dots,t_r\in \mathbb{P}^{1}$. Here $\Gamma_r$ is isomorphic to the fundamental group of $\mathbb{P}^{1}\setminus\{t_1,\dots,t_r\}$. Varying the branch points, we get a family $R(\tilde{G},\tilde{\theta},\sigma)\subset R_g$ (the image of $T$ in $R_g$ mentioned above). The paper \cite{CFGP} then gives examples of families $R(\tilde{G},\tilde{\theta},\sigma)$ for which the Zariski closure $\overline{\mathscr{P}(R(\tilde{G},\tilde{\theta},\sigma))}$ of the image under the Prym map is a Shimura curve in $A_g$.\\
	
	In the paper \cite{FG}, the first author together with P. Frediani investigated the occurrence of Shimura curves arising from families of Prym varieties of double covers ramified over more than two points. Note that in this case the Prym variety is not principally polarized in general. The paper \cite{FG} is thus a generalization of the paper \cite{CFGP} which only considers Pryms of unramified double covers or double covers ramified over at most two points. The subsequent work \cite{FGM} of the present authors together with P. Frediani investigated the same problem for higher dimensional Shimura varieties contained in these loci.\par  In this paper we generalize the aforementioned papers in two directions: We consider families of Pryms of arbitrary Galois covers of curves (not necessarily double covers) while we also get higher-dimensional as well as 1-dimensional families.  More precisely to a finite Galois covering $f:\widetilde{C}\to C$ we associate a Prym variety $P(\widetilde{C}/C)$. The Prym variety is of dimension $p=\widetilde{g}-g$, where $\widetilde{g},g$ are the genera of $\widetilde{C}, C$ respectively. Furthermore, it is an abelian variety of certain polarization type $D$, see section 3. So it determines a point in the moduli space $A_{p,D}$ of complex abelian varieties of dimension $p$ and polarization type $D$. The stack parametrizing families of the above covers of curves will be denoted by $R(H,g,r)$, where $g=g(C_t)$ is the genus of the base curve and $r$ is the number of branch points of the covering. \par We consider the following families of curves: We fix a finite group $\widetilde{G}$ and a normal subgroup $H\subseteq\widetilde{G}$. The families that we consider here are families $\widetilde{C}_t\rightarrow \mathbb{P}^{1}$ of $\widetilde{G}$-Galois covers of $\P^1$. By varying the branch points  we obtain a family $f:\sC\to T_s$ of covers of $\mathbb{P}^{1}$. This family gives rise to the family $\widetilde{C}_t\to C_t=\widetilde{C}_t/H$ in $R(H,g,r)$ and the corresponding family of Pryms $P(\widetilde{C}_t/C_t)$. The image of $T_s$ in $R(H,g,r)$, which we again denote by $T_s$, is of dimension $s-3$. The Prym map behaves well in families and we are interested in the Zariski closure of the image $Z=\overline{\mathscr{P}(T_s)}$ under the Prym map which is a subvariety of $A_{p,D}$. For computational reasons, the case where $\widetilde{G}$ is abelian is of great importance for us. Therefore, in Section \ref{abeliancovers}, we explain an alternative construction of the abelian covers of $\mathbb{P}^{1}$ than given in \cite{CFGP}. 
	Note that the Prym variety and the Prym map of abelian and metabelian covers have been studied in \cite{M20}. The $H$-action and its eigenspaces on the cohomology and also the eigenspaces of the whole group $\widetilde{G}$ acting on these spaces are useful for our computations. \par In Section \ref{Shimura} we point out that the moduli space $A_{p,D}$ has the structure of a \emph{Shimura variety}. We find families for which the subvariety $Z$ is a special (or Shimura) subvariety of $A_{p,D}$. We introduce conditions (B), (B1) and (B2) under which the subvariety $Z$ is special. Using computer algebra we investigate these condition and find 210 examples satisfying them, see the table on page 21. We also work out in detail some important examples of the table. In addition to families of abelian covers, we find some families with $\widetilde{G}$ non-abelian. Furthermore, our approach yields also higher dimensional special families of Pryms in $A_{p,D}$. Note that in \cite{CF} the authors give upper bounds for the dimension of a germ of a totally geodesic submanifold, and hence of a special subvariety in the Prym locus.  
	\section{Prym map and the Prym variety} \label{Prym}
	To a given finite covering $f:\widetilde{C}\to C$ between non-singular projective algebraic curves (or Riemann surfaces) one can associate an abelian variety, the so-called \emph{Prym variety} : $f$ induces a \emph{norm map}
	\begin{align*}
	\nm_f:\pic^0(\widetilde{C})\to \pic^0(C)\\
	\sum a_ip_i\mapsto \sum a_if(p_i)
	\end{align*}
	The Prym variety associated to $f$ is then defined as $P(f)=P(\widetilde{C}/C)=(\ker \nm_f)^0$, i.e., the connected component of the kernel of $\nm_f$ containing the identity. Identifying $\pic^0$ with the Jacobian, one sees that the canonical (principal) polarization of $\jac(\widetilde{C})$ restricts to a polarization on $P(f)$.\par 
	Classically, $f$ is a double covering which is \'etale or branched at exactly two points. In these cases, $P(f)$ is known to be principally polarized. In fact, these are the only cases in which the polarization on $P(f)$ is principal. However, the type $D$ of the polarization on $P(f)$ depends on the topological structure of the covering map $f$, see \cite{BL}. \par 
	Let $H$ be a finite group with $n=|H|$. Suppose $C$ is a compact Riemann surface of genus $g$. Let $t:=\{t_1,\dots, t_r\}$ be an $s$-tuple of distinct points in $C$. Set $U_t:=C\setminus \{t_1,\dots, t_r\}$. The fundamental group $\pi_1(U_t, t_0)$ has a presentation $\langle \alpha_1,\beta_1,\dots,\alpha_g,\beta_g, \gamma_1,\dots,\gamma_r\mid \prod_1^s\gamma_i \prod_1^g[\alpha_j,\beta_j]=1\rangle$. Here $\alpha_1,\beta_1,\dots,\alpha_g,\beta_g$ are simple loops in $U_t$ which only intersect in $t_0$, and their homology classes in $H_1(C,\Z)$ form a symplectic basis.\par 
	If $f:\widetilde{C}\to C$ is a  ramified $H$-Galois cover with branch locus $t$, set $V=f^{-1}(U_t)$. Then $f|_V:V\to U_t$ is an unramified Galois covering. Then there is an epimorphism $\theta:\pi_1(U_t, t_0)\to H$. Conversely, such an epimorphism determines a ramified Galois covering of $C$ with branch locus $t$.  The order $m_i$ of $\theta(\gamma_i)$ is called the \emph{local monodromy datum} of the branch point $t_i$. Let $m=(m_1,\dots, m_r)$. The collection $(m,H,\theta)$ is called a \emph{datum}. The Riemann-Hurwitz formula implies that the genus $\widetilde{g}$ of the curve $\widetilde{C}$ is equal to
	\begin{equation}
	2\widetilde{g}-2=|H|(2g-2+\sum_{i=1}^r (1-\frac{1}{m_i}))
	\end{equation}
	We introduce the stack $R(H,g,r)$: The Objects of $R(H,g,r)$ are couples $((C, x_1,\dots, x_r),f:\widetilde{C}\to C)$ such that
	\begin{enumerate}
		\item $(C, x_1,\dots, x_r)$ is a smooth projective $r$-pointed curve of genus $g$.
		\item $f:\widetilde{C}\to C$ is a finite cover, $H$ acts on  $\widetilde{C}$ and $f$ is $H$-invariant.
		\item the restriction $f^{gen}:f^{-1}(C\setminus\{x_1,\dots, x_r\})\to C\setminus\{x_1,\dots, x_r\}$ is an \'etale $H$-torsor.
	\end{enumerate}
	Note that $r=0$ is also possible which amounts to say that the covers $\widetilde{C}\to C$ are unramified. Moreover since our problem is insensitive to level structures, we may actually consider $R(H,g,r)$ as a coarse moduli space. As a result, we omit any assumptions on the automorphism group of the base curve $C$ whose non-triviality can be remedied either by considering the moduli stack or by imposing level structures.\par Let us denote the Jacobians of the curves $\widetilde{C}$ and $C$ respectively by $\widetilde{J}$ and $J$. Note that by definition, if $R$ is a Riemann surface,
	\[J(R)=\jac(R)=H^0(R,\omega_{R})^*/H_1(R,\mathbb{Z}).\]
	Since the finite group $H$ acts on $\widetilde{C}$ it also acts on the space of differential 1-forms $H^0(\widetilde{C},\omega_{\widetilde{C}})$ and $H_1(\widetilde{C},\mathbb{Z})$ and hence on the Jacobian $\widetilde{J}$. In particular, we denote by $\widetilde{J}^{H}$ the subgroup of fixed points of $\widetilde{J}$ under the action of $H$. The following theorem is proven in \cite{RR} (repectively, Theorem 2.5 and Proposition 3.1).
	\begin{theorem} \label{Subtorus-jacobian}
		\begin{enumerate}
			\item $f^*J=(\widetilde{J}^H)^0$.\item The map $f$ induces an isogeny $J\times P(\widetilde{C}/C)\sim \widetilde{J}$
		\end{enumerate}
	\end{theorem}
	We note that the isogeny mentioned in Theorem \ref{Subtorus-jacobian} is given by
	\begin{gather}\label{isogeny equation}
	\phi:J\times P(\widetilde{C}/C)\to \widetilde{J}\nonumber\\
	\phi(c,\widetilde{c})=f^*c+\widetilde{c}
	\end{gather}
	For a Galois covering $f:\widetilde{C}\to C$ with $((C, x_1,\dots, x_r),f:\widetilde{C}\to C)\in R(H,g,r)$ and $\deg(f)=n$, one can compute the genus $\widetilde{g}:=g(\widetilde{C})$ by the Riemann-Hurwitz formula. Using the isogeny $f^*J\times P(\widetilde{C}/C)\sim \widetilde{J}$ we see that the dimension of the Prym variety $P(\widetilde{C}/C)=P(f)$ is equal to $p=\widetilde{g}-g$. The canonical principal polarization on $\widetilde{J}$ restricts to a polarization of type $D=(1,\dots,1,n,\dots,n)$ where $1$ occurs $g-1$ times and $n$ occurs $p-(g-1)$ times if $r=0$ and $1$ occurs $g$ times and $n$ occurs $p-g$ times otherwise.\par
	Note that, it follows from Theorem \ref{Subtorus-jacobian} that if $C\cong\P^1$, then the Prym variety $P(\widetilde{C}/C)$ is isogeneous to the Jacobian $\widetilde{J}$. We will use this point in the sequel to deduce that some families are special. \par
	Let $A_{p,D}$ denote the moduli space of complex abelian varieties of dimension $p$ and polarization type $D$. More precisely, $A_{p,D}=\H_p/\Gamma_D$ is the moduli space of polarized abelian varieties of type $D$ where \par  $\H_p:=\{M\in M_p(\C)\mid ^tM=M, \im M\geq 0\}$ is the \emph{Siegel upper half space  of genus $p$} and
	\[\Gamma_D=\{R\in\gl_{2p}(\Z)\mid R\begin{pmatrix}
	
	0  &D\\
	
	-D &0
	
	\end{pmatrix}  {^tR}=\begin{pmatrix}
	
	0  &D\\
	
	-D &0
	
	\end{pmatrix}\}\]
	is an arithmetic subgroup. The above constructions behave well also in the families of curves and hence we obtain a morphism
	\begin{equation}
	\mathscr{P}=\mathscr{P}(H,g,r):R(H,g,r)\to A_{p,D}.
	\end{equation}
	We call the map $\mathscr{P}$ the \emph{Prym map of type $(H,g,r)$}. Our objective in this paper is to study the image of this map. The Prym map is even in the classical case known to be non-injective which implies that one needs to study other closely related aspects, namely the generic injectivity. \par
	By the above mentioned $H$-action on $H^0(\widetilde{C},\omega_{\widetilde{C}})$ and
	$H_1(\widetilde{C},\mathbb{Z})$,  we set:
	\begin{align} \label{plus minus}
	H^0(\widetilde{C},\omega_{\widetilde{C}})_+=H^0(\widetilde{C},\omega_{\widetilde{C}})^{H}(\cong H^0(C,\omega_{C})) \text{ and }\\
	H^0(\widetilde{C},\omega_{\widetilde{C}})_-=H^0(\widetilde{C},\omega_{\widetilde{C}})/H^0(\widetilde{C},\omega_{\widetilde{C}})^+=
	\bigoplus\limits_{\chi\in\irr(H)\setminus\{1\}}H^0(\widetilde{C},\omega_{\widetilde{C}})^{\chi}
	\end{align}
	Notice that $H^0(\widetilde{C},\omega_{\widetilde{C}})=H^0(\widetilde{C},\omega_{\widetilde{C}})_+\oplus H^0(\widetilde{C},\omega_{\widetilde{C}})_-$. \par The following lemma is then an immediate consequence of Theorem \ref{Subtorus-jacobian} above.
	\begin{lemma} \label{prymvar}
		Let $f:\widetilde{C}\to C$ be a Galois covering, then
		\begin{equation} \label{Prymdef}
		P(\widetilde{C}/C)={H^0(\widetilde{C},\omega_{\widetilde{C}})_-}^*/H_1(\widetilde{C},\mathbb{Z})_-
		\end{equation}
	\end{lemma}
	\section{Galois coverings} \label{abeliancovers}
	\subsection{Generalities}
	Let us summarize some general facts about Galois coverings of curves. Let $\widetilde{C},C$ be complex smooth projective algebraic curves (equivalently Riemann surfaces) and let $f\colon\widetilde{C}\to C$ be a Galois covering of degree $n$. By this we mean precisely that there exists a finite group $H$ with $|H|=n$ together with a faithful action of $H$ on $\widetilde{C}$ such that $f$ realizes $C$ as the quotient of $\widetilde{C}$ by $H$. Consider the ramification and branch divisors $R,B$ of $f$. Note
	that $R$ consists precisely of the points in $\widetilde{C}$ with non-trivial stabilizers under the action of $H$. The deck transformation group $\deck(\widetilde{C}/ C)$, i.e., the group of those automorphisms of $\widetilde{C}$ that are compatible with $f$ is isomorphic to the Galois group $H$ and acts transitively on each fiber $f^{-1}(x)$. If $y\in \widetilde{C}$ is a ramification point with ramification index $e$, then so are all points in the fiber $f^{-1}(f(y))$. Moreover, the stabilizers of these points in $\deck(\widetilde{C}/ C)\cong H$ are conjugate cyclic subgroups of $\deck(\widetilde{C}/ C)$, see \cite{Sz}, Proposition 3.2.10. In particular the stabilizer of a point in $\widetilde{C}$ is trivial, if and only if that point is \emph{not} a ramification point. The stabilizer $H_y$ of a point $y\in \widetilde{C}$ is also referred to as  the \emph{inertia subgroup} of $y$.
	\subsection{Galois covers of $\mathbb{P}^{1}$ and their Prym map}
	Let $\widetilde{G}$ be a finite group and $\widetilde{f}:\widetilde{C}\to \mathbb{P}^1$ a finite $\widetilde{G}$-Galois covering ramified over the branch points
	$\br(\widetilde{f})=\{t_1,\dots,t_s\}\subset\mathbb{P}^1$ as in introduction. Let $\Gamma_s:=\pi_1(\mathbb{P}^1\setminus\br(\widetilde{f}))=\langle\gamma_{1},\dots, \gamma_{s}|\gamma_{1}\cdots\gamma_{s}=1\rangle$, where $\gamma_{j}$ corresponds to a loop winding around $z_{j}$. Such a $\widetilde{G}$-Galois covering is determined by an
	epimorphism $\widetilde{\theta}_s:\Gamma_s\to\widetilde{G}$ (See \cite{Vo}, Theorem 5.14). The local monodromy around the branch point $t_{j}$ is given by $\widetilde{\theta}_s(\gamma_{j})$. The set of ramification points $\ram(\widetilde{f})$ consists
	precisely of the points in $\widetilde{C}$ with non-trivial stabilizers under the action of $\widetilde{G}$. As we assume that the cover $\widetilde{f}$ is Galois we have that
	$\br(\widetilde{f})=\widetilde{f}(\ram(\widetilde{f}))$ and $\widetilde{f}^{-1}(\br(\widetilde{f}))=\ram(\widetilde{f})$. Varying the branch points $\{t_1,\dots,t_s\}$ yields a family of $\widetilde{G}$-covers of $\P^1$. If $H$ is a normal subgroup of $\tilde{G}$, then to the quotient cover $f:\widetilde{C}\to C=\widetilde{C}/H$ one can associate a Prym variety as in the beginning of Section \ref{Prym}. The following definition is central in this paper.
	\begin{definition}\label{Prym datum general}
		A \emph{Prym datum} (of type $(H,g,r)$, compare \cite{CFGP}, Definition 3.1) is a triple $(\tilde{G}, \tilde{\theta}_s,H)$ where $\tilde{G}$ is a finite group, $\widetilde{\theta}_s:\Gamma_s\to\widetilde{G}$ is an epimorphism as above and $H$ is a normal subgroup of $\tilde{G}$, such that the quotient $f:\widetilde{C}\to C=\widetilde{C}/H$ is in $R(H,g,r)$.
	\end{definition}
	Let $\widetilde{G}$ be a finite group and let $\widetilde{C}\to \mathbb{P}^{1}$  be a $\widetilde{G}$-Galois covering of $\mathbb{P}^{1}$ with the Prym datum $(\tilde{G}, \tilde{\theta}_s,H)$. Set $V=H^0(\widetilde{C},\omega_{\widetilde{C}})$ and let $V=V_{+}\oplus V_{-}$ be the decomposition into $H$-invariant and $H$-anti-invariant parts as in \ref{abelian decomposition}. There is also the corresponding Hodge decomposition $H^1(\widetilde{C},\mathbb{C})_-=V_{-}\oplus \overline{V}_{-}$. Set $\Lambda=H_1(\widetilde{C},\mathbb{Z})_-$. The associated Prym variety is by definition the following abelian variety.
	\begin{equation}\label{Prymdef1}
	P(\widetilde{C}/C)=V^{*}_{-}/\Lambda,
	\end{equation}
	see \cite{BL} for more details.
	\subsection{Abelian covers and their invariants}\label{Abelian Galois covers}
	Let $f:\widetilde{C}\to C$ be a $H$-Galois cover with $H$ finite abelian branched above the points $x_1,\dots, x_r$. Since the group $H$ is abelian, the inertia group above a branch point $x_i$ is independent of the chosen ramification point and we denote it by $H_i$.  Then $f_*\mathcal{O}_{\widetilde{C}}=\bigoplus\limits_{\chi \in G^*}L^{-1}_{\chi}$, where each $L_{\chi}$ is an invertible sheaf on $C$ on which $G$ acts by character $\chi$. So in particular, the invariant summand $L_1$ is ismorphic to $\mathcal{O}_C$. The algebra structure on $f_*\mathcal{O}_{\widetilde{C}}$ is given by the ($\mathcal{O}_C$-linear) multiplication rule $m_{\chi,\chi^{\prime}}:L^{-1}_{\chi}\otimes L^{-1}_{\chi^{\prime}}\to L^{-1}_{\chi\chi^{\prime}}$ and compatible with the action of $G$. The choice of a primitive $n$-th root of unity $\xi$ amounts to giving a map $\{1,\dots, r\}\to H$, the image $h_i$ of $i$ under which is the generator of the inertia group $H_i$ that is sent to $\xi^{n/n_i}$ by $\chi_i$. The line bundles $L_{\chi}$ and divisors $x_i$ each labelled with an element $h_i$ as described above are called the building data of the cover. In fact given these data, one can construct an abelian cover which has these bundles as its building data, see \cite{P1}. We have the following eigenspace decompositions for the abelian cover.
	\begin{align}
	H^0(\widetilde{C},\omega_{\widetilde{C}})=H^0(C,f_*\omega_{\widetilde{C}})=H^0(C,\oplus(\omega_{C}\otimes L_{\chi^{-1}}))=\oplus_{\chi\in H^*} H^0(C,\omega_{C}\otimes L_{\chi^{-1}})
	\end{align}
	where the second equality is due to the equality $(f_*\omega_{\widetilde{C}})^{\chi}=\omega_{C}\otimes L_{\chi^{-1}}$ for abelian covers, see \cite{P1}. In view of the above equalities, one obtains
	\begin{align}  
	H^0(\widetilde{C},\omega_{\widetilde{C}})_+=H^0(C,\omega_{C})\nonumber\\
	H^0(\widetilde{C},\omega_{\widetilde{C}})_-=\oplus_{\chi\in H^*\setminus\{1\}}
	H^0(\widetilde{C},\omega_{\widetilde{C}})^{\chi})=\oplus_{\chi\in H^*\setminus\{1\}} H^0(C,\omega_{C}\otimes L_{\chi^{-1}})  \label{abelian decomposition}
	\end{align}
	\subsection{Prym varieties of abelian covers}
	In this subsection we explain the constructions in section \ref{Prym} for an abelian group $H$ based on the constructions of section \ref{Abelian Galois covers}. So let $f:\widetilde{C}\to C$ be a $H$-Galois cover of $C$, with $H$ a finite abelian group. Recall the equivalent description of the Prym variety given in Lemma \ref{Prymdef}.
	For a Galois covering $f:\widetilde{C}\to C$ of a curve $C$ of genus $g$ as above, one can compute the genus $\widetilde{g}:= g(\widetilde{C})$ by the Riemann-Hurwitz formula. \par By what we already said in subsection \ref{Abelian Galois covers}, we have that
	\[P=P(\widetilde{C}/C)=\oplus_{\chi\in H^*\setminus\{1\}} H^0(C,\omega_{C}\otimes L_{\chi^{-1}})/\oplus_{\chi\in H^*\setminus\{1\}} H_1(\widetilde{C},\mathbb{Z})^{\chi},\]
	by virtue of \ref{abelian decomposition} and Lemma \ref{prymvar}.
	\subsection{Abelian covers of $\mathbb{P}^{1}$ and their Prym map}\label{Abelian covers of line}
	In this section, we follow closely \cite{CFGP} and also \cite{MZ} whose notations come mostly from \cite{W}. More details about abelian coverings and Prym varieties can be consulted from these two references respectively. For the latter, \cite{BL} is also a comprehensive reference. \par An abelian Galois cover of $\mathbb{P}^{1}$ is determined by a collection of equations in the following way: \par Consider an $m\times s$ matrix $A=(r_{ij})$ whose entries $r_{ij}$ are in $\mathbb{Z}/N\mathbb{Z}$ for some $N\geq 2$. Let $\overline{\mathbb{C}(z)}$ be the algebraic closure of $\mathbb{C}(z)$. For each $i=1,...,m,$ choose a function $w_{i}\in\overline{\mathbb{C}(z)}$ with
	\begin{equation}\label{equation abelian}
	w_{i}^{N}=\prod_{j=1}^{s}(z-z_{j})^{\widetilde{r}_{ij}}\text{ for }i=1,\dots, m,
	\end{equation}
	in $\mathbb{C}(z)[w_{1},\dots,w_{m}]$. Here $\widetilde{r}_{ij}$ is the lift of $r_{ij}$ to $\mathbb{Z} \cap [0,N)$ and $z_j\in \C$ for $j=1,2,\dots, s$. Notice that \ref{equation abelian} gives in general only a singular affine curve and we take a smooth projective model associated to this affine curve. We impose the condition that the sum of the columns of $A$ is zero (when considered as a vector in $(\mathbb{Z}/N\mathbb{Z})^m$). This implies that the cover given by \ref{equation abelian} is \emph{not} ramified over the infinity. We call the matrix $A$, the matrix of the covering. We also remark that all operations with rows and columns will be carried out over the ring $\mathbb{Z}/N\mathbb{Z}$, i.e., they will be considered modulo $N$. The local monodromy around the branch point $z_{j}$ is given by the column vector $(r_{1j},\dots, r_{mj})^{t}$ and so the order of ramification over $z_{j}$ is $\frac{N}{\gcd(N,\widetilde{r}_{1j},\dots,\widetilde{r}_{mj})}$. Using this and the Riemann-Hurwitz formula, the genus $g$ of the cover can be computed by:
	\begin{equation}
	g=1+d(\frac{s-2}{2}-\frac{1}{2N}\sum_{j=1}^{s}\gcd(N,\widetilde{r}_{1j},\dots,\widetilde{r}_{mj})),
	\end{equation}
	where $d$ is the degree of the covering which is equal, as pointed out above, to the column span (equivalently row span)  of the matrix $A$. In this way, the Galois group $\widetilde{G}$ of the covering will be a subgroup of $(\mathbb{Z}/N\mathbb{Z})^{m}$. Note also that this group is isomorphic to the column span of the above matrix.
	\begin{remark} \label{rowspan}
		Consider two families of abelian covers with matrices $A$ and $A^{\prime}$ over the same
		$\mathbb{Z}/N\mathbb{Z}$. If $A$ and $A^{\prime}$ have equal row spans then the two families are isomorphic. For more details, see \cite{W} or \cite{MZ}.
	\end{remark}
	\begin{remark} \label{abeliangroupcharacter}
		Let $G$ be a finite abelian group, then the character group $G^*=\Hom(G,\mathbb{C}^{*})$ is isomorphic to $G$. To see this, first assume that $G=\Z/N\Z$ is a cyclic group. Fix an isomorphism between $\mathbb{Z}/N\Z$ and the group of $N$-th roots of unity in $\mathbb{C}^{*}$ via $1\mapsto \exp(2\pi i/N)$. Now the group $G^*$ is isomorphic to this latter group via $\chi\mapsto \chi(1)$. In the general case, $G$ is a product of finite cyclic groups, so this isomorphism extends to an isomorphism $\varphi_G: G \xrightarrow{\sim} G^*$. In the sequel, we use this isomorphism frequently to identify elements of $G$ with its characters.
	\end{remark}
	For our applications, with notations as in the previous pages, we fix an isomorphism of $\widetilde{G}$ with a product of $\mathbb{Z}/n\Z$'s and an embedding of $\widetilde{G}$ into $(\mathbb{Z}/N\Z)^{m}$.\par Let $l_{j}$ be the $j$-th column of the matrix $A$. As mentioned earlier, the group $\widetilde{G}$ can be realized as the column span of the matrix $A$. Therefore we may assume that $l_{j}\in \widetilde{G}$. For a character $\chi$, $\chi(l_{j})\in \C^{*}$ and since $\widetilde{G}$ is finite $\chi(l_{j})$ will be a root of unity. Let $\chi(l_{j})=\exp(2\alpha_{j}\pi i/N)$, where $\alpha_{j}$ is the unique integer in $[0,N)$ with this property. Equivalently, the $\alpha_{j}$ can be obtained in the following way: let $n\in G\subseteq (\Z/N\Z)^{m}$ be the element corresponding to $\chi$ under the above isomorphism.
	We regard $n$ as an $1\times m$ matrix. Then the matrix product $n\cdotp A$ is meaningful and $n\cdotp A=(\alpha_{1},\dots,\alpha_{r})$.
	Here all of the operations are carried out in $\Z/N\Z$ but the $\alpha_{j}$ are regarded as    integers in $[0,N)$. Furthermore we set $\tilde\alpha_j = \sum_{i=1}^m n_i \tilde{r}_{ij} \in \Z$ (but $\tilde\alpha_j$ is not necessarily in $\Z\cap [0,N)$).
	Using the above facts, we occasionally consider a character of $\widetilde{G}$ as an element of this group without referring to isomorphism $\varphi_{\widetilde{G}}$.
	
	Let us denote by $\omega_X$ the canonical sheaf of $X$.
	Similar to the case of $\pi_{*}(\mathcal{O}_X)$, the sheaf $\pi_{*}(\omega_X)_{\chi}$ decomposes according to the action of $\widetilde{G}$.
	For the line bundles $L_{\chi}$ corresponding to the character $\chi$ associated to the element $a\in \widetilde{G}$ and $\pi_{*}(\omega_X)_{\chi}$ we have the following result proven in \cite{MZ}. 
	\begin{lemma} \label{eigenbundleformula}
		With notations as above $L_{\chi}=\mathcal{O}_{\mathbb{P}^{1}}(\displaystyle\sum_{1}^{s}\langle\frac{\tilde\alpha_j}{N}\rangle)$, where $\langle x\rangle$ denotes the fractional part of the real number $x$ and
		\[\pi_{*}(\omega_X)_{\chi}= \omega_{\mathbb{P}^{1}} \otimes L_{\chi^{-1}}=\mathcal{O}_{\mathbb{P}^{1}}(-2+\sum_{1}^{s}\langle -\frac{\alpha_j}{N}\rangle).\]
	\end{lemma}
	Let $n\in\widetilde{G}$ be the element $(n_1,\dots, n_m)\in\widetilde{G}\subset (\mathbb{Z}/N\mathbb{Z})^{m}$. By Lemma \ref{eigenbundleformula}, $\dim H^0(\widetilde{C},\omega_{\widetilde{C}})_{n}=-1+\displaystyle \sum_{j=1}^{s}\langle-\frac{\alpha_{j}}{N}\rangle$. A basis for the $\mathbb{C}$-vector space $H^0(\widetilde{C},\omega_{\widetilde{C}})$ is given by the forms
	\begin{equation}
	\omega_{n,\nu}=z^{\nu} w_{1}^{n_1}\cdots w_{m}^{n_m}\displaystyle \prod_{j=1}^{s} (z-z_j)^{\lfloor -\frac{\tilde{\alpha_j}}{N}\rfloor}dz.
	\end{equation}
	Here $0\leq\nu\leq -1+\displaystyle \sum_{j=1}^{s}\langle-\frac{\alpha_{j}}{N}\rangle$. The fact that the above elements constitute a basis can be seen in \cite{MZ}, proof of Lemma 5.1, where the dual version for $H^1(C,\mathcal{O}_C)$ is proved. \par
	The general method of our later computations in Section \ref{examples} is as follows: We remark that if $n=(n_1,\dots,n_m)\in\widetilde{G}=\mathbb{Z}_{d_1}\times\cdots\times\mathbb{Z}_{d_m}\subset(\mathbb{Z}/N\mathbb{Z})^{m}$, we consider the $n_i\in [0,N)$ and their sum as integers. \par The action of the abelian subgroup $H$ is naturally inherited from that of $\tilde G$ and the latter is described as follows: Let $g=(g_1,\dots, g_m)\in\tilde G$ and write $\ord g_i=v_i$. Then the action of $g$ on each $w_i$ is given by $g\cdotp w_i=\xi_{v_i}w_i$, where $\xi_{v_i}$ denotes a $v_i$-th primitive root of unity.

	With this notation, $H^0(\widetilde{C},\omega_{\widetilde{C}})_+$, i.e., the group of $H$-invariant differential forms is the set of all $\omega_{n,\nu}$ with $\sum n_i/a_i\in\mathbb{Z}$ for all $h=(h_1,\dots,h_m)\in H$ (with $a_i=\ord h_i$).

	The eigenspace $H^0(\widetilde{C},\omega_{\widetilde{C}})_-$ is then given by the complement, i.e., the set of all $\omega_{n,\nu}$ for whom there exists $h=(h_1,\dots,h_m)\in H$ such that $\sum n_i/a_i\notin\mathbb{Z}$.\par Families of abelian covers of $\P^1$ can be constructed as follows:  Let $T_{s}\subset (\mathbb{A}_{\C}^{1})^{s}$ be the complement of the big diagonals, i.e., $T_{s}= \{(z_{1},\dots,z_{s})\in(\mathbb{A}_{\C}^{1})^{s}\mid z_{i}\neq z_{j} \forall i\neq j \}$. Over this affine open set we define a family of abelian covers of $\mathbb{P}^{1}$ by the equation \ref{equation abelian} with branch points $(z_{1},\dots,z_{s})\in T_{s}$ and $\widetilde{r}_{ij}$ the lift of $r_{ij}$ to $\mathbb{Z}\cap[0,N)$ as before. Varying the branch points we get a family $f:\tilde{\sC}\to T_s$ of smooth projective curves over $T_s$ (viewed as a complex manifold of dimension $s-3$) whose fibers $\tilde{C}_t$ are abelian covers of $\mathbb{P}^{1}$ introduced above.
	\section{Shimura subvarieties} \label{Shimura}
	Recall from section \ref{Prym} that $A_{p,D}=\H_p/\Gamma_D$ is the moduli space of polarized abelian varieties of type $D$, where $\H_p:=\{M\in M_p(\C)\mid ^tM=M, \im M\geq 0\}$, is the Siegel upper half space  of genus $p$ and $\Gamma_D=\{R\in\gl_{2p}(\Z)\mid R\begin{pmatrix}
	
	0  &D\\
	
	-D &0
	
	\end{pmatrix}  {^tR}=\begin{pmatrix}
	
	0  &D\\
	
	-D &0
	\end{pmatrix}\}$ is an arithmetic subgroup. Note that $\H_p=\gsp_{2p}(\R)/K$, where $\gsp_{2p}$ is the standard $\Q$-group of symplectic similitudes on the standard symplectic $\Q$-space $\Q^{2p}$ and $K$ is a maximal compact subgroup. So $A_{p,D}$ can be written as a double quotient $\Gamma_D\backslash\gsp_{2p}(\R)/K$. Such double quotients are called Shimura variety and their structure has beeen studied extensively. A special (or Shimura) subvariety of $A_{p,D}$ is then an algebraic subvariety of the form $Y=\Gamma\backslash\D\hookrightarrow A_{p,D}$ induced by an injective homomorphism $\L\hookrightarrow \Sp_{2p}(\R)$ of algebraic groups. In particular, it is a totally geodesic subvariety.\par Let $\widetilde{G}$ be a finite group and consider a family $\widetilde{\sC}\to T_s$ whose fibers $\widetilde{C}_t$ are $\widetilde{G}$-Galois coverings of $\mathbb{P}^{1}$ with a fixed Prym datum $\Sigma\coloneqq(\tilde{G}, \tilde{\theta}_s,H)$. Associating to $t\in T_s$ the class of the pair $((C_t,x_1,\dots,x_r),\pi_t:\widetilde{C}_t\to C_t)$ gives a map $T_s\to R(H,g,r)$ with discrete fibers. We denote the image of this map by $R(\Sigma)$. It follows that $R(\Sigma)$ is a subvariety of dimension equal to $s-3$, see also \cite{CFGP}, p. 6. As in the last section, set $V_t=H^0(\widetilde{C}_t,\omega_{\widetilde{C}_t})$ and let $V_t=V_{+,t}\oplus V_{-,t}$ be the decomposition under the action of $H$. There is also the corresponding Hodge decomposition $H^1(\widetilde{C}_t,\mathbb{C})_-=V_{-,t}\oplus \overline{V}_{-,t}$. Set $\Lambda_t=H_1(\widetilde{C}_t,\mathbb{Z})_-$. The associated Prym variety is by \ref{Prymdef1}, $P(\widetilde{C}_t/C_t)=V^{*}_{-,t}/\Lambda_t$, an abelian variety of dimension $p=\widetilde{g}-g$. So we have a map $R(\Sigma)\xrightarrow{\mathscr{P}} A_{p,D}$.\par In this paper, we are interested in determining whether the subvariety $Z=\overline{\mathscr{P}(R(\Sigma))}\subset A_{p,D}$ is a special or Shimura subvariety. We remark that one can, possibly after replacing $T_s$ with a suitable finite cover, endow the abelian scheme with a level structure and hence consider the resulting map $T_s\to A_{p,D,n}$ to the fine moduli space. Note that the answer to the above question is independent of the level structure or the polarization. Alternatively, we can consider $A_{p,D}$ as a coarse moduli space. The Prym varieties of the fibers of the family $\widetilde{C}_t\to T_s$ fit into a family $P\rightarrow T_s$ which is an abelian scheme over $T_s$ that admits naturally an action of the group ring $\mathbb{Z}[\widetilde{G}]$. This action defines a Shimura subvariety of PEL type $P(\widetilde G)$ in $A_{p,D}$ (or in the stack $\mathcal{A}_{p,D}$) that contains $Z$. The following constrcution of the subvariety $P(\widetilde G)$ is adapted for the case of  Prym varieties from \cite{MO}, see also the paper \cite{M10} and also \cite{FGP} for a different approach. Fix a base point $t \in T_s$ and let $(P_t,\lambda)$ be the corresponding Prym variety with $\lambda$ as its polarization of type $D$. Let $(V_{\mathbb{Z}},\psi)$ be as above. We fix a symplectic similitude $\sigma :H^{1}(P_t, \mathbb{Z})\to V_{\mathbb{Z}}$. Let $F=\mathbb{Q}[\widetilde{G}]$. The group $\widetilde{G}$ acts on $H^0(\widetilde{C}_t,\omega_{\widetilde{C}})_-$ and thereby on the Prym variety $P(\widetilde{C}_t/C_t)$. We therefore view $H^0(\widetilde{C}_t,\omega_{\widetilde{C}})_-$ as an $F$-module. Via $\sigma$, the Hodge structure on $H^1(P_t,\mathbb{Q})=H^1(\widetilde{C}_t,\mathbb{Q})_-$ corresponds to a point $y\in\H_p$ and one obtains the structure of an $F$-module on $V_{\mathbb{Q}}$. $F$ is isomorphic to a product of cyclotomic fields and is equipped with a natural involution $*$ which is complex conjugation on each factor. The polarization $\psi$ on $V_{\mathbb{Q}}$ satisfies
	\[\psi(bu,v)=\psi(u,b^{*}v) \text{ for all } b\in F \text{ and } u,v \in V.\]
	Define the subgroup $N$ as in \cite{MO}:
	\begin{equation}\label{subgroup}
	N= \gsp(V_{\mathbb{Q}}, \psi)\cap \gl_{F}(V_{\mathbb{Q}}).
	\end{equation}
	If $h_{0}: \mathbb{S}\rightarrow \gsp_{2p, \mathbb{R}}$ is the Hodge structure on $V_{\mathbb{Z}}=H^{1}(P_{t},\mathbb{Z})$ corresponding to the point $y \in \H_p$, then by the above $F$-action this homomorphism factors through the subvariety $N_{\mathbb{R}}$. 
	As $A_{p,D}$ has the structure of a Shimura variety, one can talk about its Shimura (or special) subvarieties. Let $L=\gsp_{2p}$. Define the subset $Y_N\subseteq\H_p$ as follows.
	\[Y_N=\{h: \mathbb{S}\rightarrow L_{2p, \mathbb{R}}| h \text{ factors through } N_{\mathbb{R}}\}.\]
	The point $y$ lies in $Y_{N}$ and there is a connected component $Y^{+}\subseteq Y_{N}$ which contains $y$. We define $P(\widetilde{G})$ to be the image of $Y^{+}$ under the map
	\[\H_p\to L(\mathbb{Z})\setminus\H_p\cong
	L(\mathbb{Q})\setminus\H_p\times
	L(\mathbb{A}_f)/L(\widehat{\mathbb{Z}})\cong A_{p,D}(\mathbb{C}).\]
	If $Y^{+}$ is a connected component of $Y_N$ and $\gamma K_n\in L(\mathbb{A}_f)/K_n$, the image of $Y^{+}\times \{\gamma K_n\}$ in $A_{p,D}$ is an algebraic subvariety. We define a \emph{Shimura subvariety} as an algebraic subvariety $S$ of $A_{p,D}$
	which arises in this way, i.e., there exists a connected component
	$Y^+\subset Y_N$ and an element $\gamma K_n\in
	L(\mathbb{A}_f)/K_n$ such that $S$ is the image of $Y^{+}\times
	\{\gamma K_n\}$ in $A_{p,D}$.\par For $t=(z_{1},\dots,z_{s})\in T_s$, let
	$((C_t,x_1,\dots,x_r),\pi_t:\widetilde{C}_t\to C_t)\in R(H,g,r)$ be the covering corresponding to
	$t$. For this $t$, consider the Hodge decomposition
	$H_1(\widetilde{C}_t,\mathbb{C})_-=V_{-,t}\oplus\overline{V}_{-,t}$
	which corresponds to a complex structure on
	$H_1(\widetilde{C}_t,\mathbb{R})_-$. We therefore get a point
	$f(t)\in\H_{p}$. Indeed we obtain a morphism
	$f:T_s\to\H_{p}$ and the following commutative diagram.
	\begin{equation} \label{normal diag}
	\begin{tikzcd} 
	T_s \arrow{r}{f} \arrow[swap]{d}{\iota^0} & \H_{p} \arrow{d}{\iota} \\
	R(\Sigma) \arrow{r}{\mathscr{P}} & A_{p,D}
	\end{tikzcd}
	\end{equation}
	It follows by construction of $P(\widetilde{G})$ that $Z\subseteq P(\widetilde{G})$. As we remarked earlier, the Prym map is not in general injective. In order to conclude the equality $Z= P(\widetilde{G})$ and hence the speciality of $Z$, we still need to assure that the differential of the Prym map on $R(\Sigma)$ is injective, whence $\dim R(\Sigma)=\dim \mathscr{P}(R(\Sigma))$. For this purpose, set $V=H^0(\widetilde{C},\omega_{\widetilde{C}})=V_+\oplus V_-$ and likewise $W=H^0(\widetilde{C},\omega^{\otimes 2}_{\widetilde{C}})=W_+\oplus W_-$. The multiplication map $m:S^2V\to W$ is the codifferential of the Torelli map and the codifferential of the Prym map at a given point coincides with the restriction of the multiplication map $m$ to $S^2V_-$. Therefore, we need that the following map is an isomorphism 
	\begin{equation}
	\label{nostramram}
	\tag{B} m : (S^2 V_-)^{\widetilde{G}} \to W_+^{\widetilde{G}}.
	\end{equation}
	This implies the condition
	\begin{equation}
	\label{condA}
	\tag{A} \dim (S^2V_-)^{\widetilde{G}} = s-3.
	\end{equation}
	A sufficient condition ensuring (B) is 
	\begin{equation}
	\label{condB1}
	\tag{B1} (S^2V_-)^{\widetilde{G}} \cong Y_1 \otimes Y_2
	\end{equation}
	where $\dim Y_1 =1$, $\dim Y_2 = s-3$. 
	
	Notice that the condition (B1) implies the condition (B). Another sufficent condition ensuring (B) is the following. Suppose there is an isogeny decomposition of the prym variety as follows 
	\begin{equation*}\tag{B2}\label{condB2}
	P(\widetilde G)\sim A\times  JC',
	\end{equation*}
	where $A$ is a fixed abelian variety and $JC'$ is the jacobian of a curve $C':=\widetilde C/K$ defined as a quotient of $\widetilde C$ by a normal subgroup $K\lhd\widetilde G$, such that the Galois cover $C'\to\P^1=C'/(\widetilde G/K)$ is branched in $r$ points and this family satisfies condition $(*)$ of \cite{FGP}, hence it gives rise to a special subvariety.
	Therefore, since $A$ is fixed and $JC'$ moves in a Shimura family, then the family of the pryms $P(\widetilde G)$ yields a special subvariety too, see \cite [Thm. 3.8]{FGS}. We have
	\begin{theorem}
		If the condition (B) holds for some $t\in T_s$, then the subvariety $Z$ is a special
		subvariety.
	\end{theorem}
	\begin{proof}
		Let $N$ be the subgroup in \ref{subgroup}. If $Y^{+}$ is a connected
		component of $Y_N$ whose image in $A_{p,D}$ is $P(\widetilde{G})$, then the
		assumption implies that $\dim Y^{+}=\dim T_s=s-3$. As the vertical
		rows in \ref{normal diag} are discrete, one concludes that $\dim\mathscr{P}(
		R(\Sigma))=\dim P(\widetilde{G})=s-3$. This together with the fact that
		$Z\subseteq P(\widetilde{G})$ implies that $Z=P(\widetilde{G})$.
	\end{proof}
When $\tilde G$ is abelian, the following Lemma computes the dimension of $P(\widetilde{G})$. 
	\begin{lemma} \label{dim P(G)}
		Let $d_n=H^{1,0}( P(\widetilde{G}))_n=H^0(\widetilde{C},\omega_{\widetilde{C}})_{-,n}$, then 
		\[\dim P(\widetilde{G})= \sum_{2n\neq 0} d_{n}d_{-n}+\frac{1}{2}\sum_{2n=0}d_{n}(d_{n}+1).\] 
	\end{lemma}
	Note that $2.0=0$ in $\widetilde G$, so in fact the second sum in the right hand side of the above equality is always meaningful and if $|\widetilde{G}|$ is an odd number it will be zero.
	\begin{proof} We calculate $\dim T_{y}(Y_{N})$ at the point $y\in\H_{p}$. The dimension of the tangent space of $P(\widetilde{G})$ at the point $y$ will be equal to this number. To compute $\dim T_{y}(\H_{p})$, we first remark that the polarization induces a perfect pairing $\overline{\phi}:H^{1,0}\times V_{\mathbb{C}}/H^{1,0}\rightarrow\C$. Then the tangent bundle $T_{y}(\H_{p})$ can be identified with
		\[\Hom^{sym}(H^{1,0}, V_{\mathbb{C}}/H^{1,0}):=\{\beta : H^{1,0}\rightarrow V_{\mathbb{C}}/H^{1,0} \mid\overline{\phi}(v,\beta(v^{\prime}))=\overline{\phi}(v^{\prime},\beta(v))\forall v,v^{\prime} \in H^{1,0}\},\]
		i.e., the elements of $T_{y}(\H_{p})$ that are their own dual via the isomorphisms induced by $\overline{\phi}$.  For a more detailed discussion, see \cite{MO}. Furthermore notice that $V_{\mathbb{C}}/H^{1,0}=H^{0,1}$ and that $\overline{\phi}$ respects the Galois group action, namely it reduces to $\overline{\phi}_n:H^{1,0}_n\times H^{0,1}_{-n}\rightarrow\C$ for every character $n$ of $\widetilde{G}$. The subspace $T_{y}(Y_{N})\subset T_{y}(\H_{p})$ consists therefore of $\beta\in\Hom^{sym}(H^{1,0}, V_{\mathbb{C}}/H^{1,0})$ (symmetric with respect to $\overline{\phi}$) that respect the $F$-action on $V$, that is, are $F_{\mathbb{C}}$-linear. Any such $\beta$ can be written as the sum $\sum \beta_{n}$, where $\beta_{n}:H^{1,0}_{\mathbb{C},n}\rightarrow H^{0,1}_{\mathbb{C},n}$ is the induced action on the eigenspaces. These $\beta_{n}$ should satisfy the relation
		\[\overline{\phi}_{n}(v, \beta_{-n}(v^{\prime}))=\overline{\phi}_{-n}(v^{\prime}, \beta_{n}(v)).\] 
		The perfect pairing $\overline{\phi}_{n}$ gives a duality between $H^{1,0}_{\mathbb{C},n}$ and $H^{0,1}_{\mathbb{C}, (-n)}$. So we have a duality between $\beta_{n}$ and $\beta_{-n}$ if $n\neq -n$ in $\widetilde{G}$. If $n= -n$ in $\widetilde{G}$, i.e., if $2n=0$ in $\widetilde{G}$ this gives a self duality for $\beta_{n}$. Therefore $\dim T_{y}(Y_{N})$ is equal to
		$\sum_{2n\neq 0} d_{n}d_{-n}+\frac{1}{2}\sum_{2n=0}d_{n}(d_{n}+1).$
	\end{proof}
	Note that the above proof implies that $\beta=\sum\beta_{n}\in\sym^2(H^{1,0}_{\mathbb{C}})^G$. In particular, it follows that $\dim P(\widetilde{G})=\dim (S^2H^0(\widetilde{C},\omega_{\widetilde{C}})_-)^{\widetilde{G}}$ see also \cite{FGP}, Theorem 3.6. 
	
	\section{Examples} \label{examples}
	In this section, we work out some the details of some of the examples given in the table on page 21. 
	\begin{itemize}
		\item 
		Consider the family given by the monodromy data $(6,(1,3,4,4))$, i.e., the family $w^{6}=(z-z_1)(z-z_2)^3(z-z_3)^4(z-z_4)^4$. This family has Galois group $\mathbb{Z}_{6}$ and fiber genus $3$. The quotient by the subgroup $\Z_3$ gives rise to a triple cover $\widetilde{C}_t\to C_t$ which is totally ramified at 5 points, so that the Prym family is contained in $R_{3,[5]}$ in the notation of \cite{CFGP}. This family corresponds to the example with data $(r,\tilde g, \#)=(4,3,2)$ in the table. The quotient curve $C_t$ corresponds to the $\Z_2$-covering $w^{2}=(z-z_1)(z-z_2)$, so that $C\cong\P^1$. Therefore $P(\widetilde{C}_t/C_t)$ is isogeneous to $J(\widetilde{C}_t)$. By the results of \cite{M10}, this latter family is a special family of Jacobians and hence the family of Pryms is also special. \\
		
		Alternatively, one could use the special subvariety $P(\widetilde{G})$ to prove that the family is special.
		
		The automorphism $\sigma$ of order 3 corresponds to the automorphism $w\mapsto \xi_3w$, where $\xi_3$ is a primitive 3rd root of unity. Using this action, we compute the eigenspace $H^0(\widetilde{C},\omega_{\widetilde{C}})_-$. We have that
		\[H^0(\widetilde{C},\omega_{\widetilde{C}})_-=H^0(\widetilde{C},\omega_{\widetilde{C}})_1\oplus H^0(\widetilde{C},\omega_{\widetilde{C}})_2 \oplus H^0(\widetilde{C},\omega_{\widetilde{C}})_4\oplus H^0(\widetilde{C},\omega_{\widetilde{C}})_5,\]
		where $H^0(\widetilde{C},\omega_{\widetilde{C}})_i$ is the eigenspace w.r.t the character $i\in \Z_{6}$. For a cyclic cover, these are standard to compute, e.g. \cite{M10}, p.799. We have that $\dim H^0(\widetilde{C},\omega_{\widetilde{C}})_1=\dim H^0(\widetilde{C},\omega_{\widetilde{C}})_5=1$, $\dim H^0(\widetilde{C},\omega_{\widetilde{C}})_2=0, \dim H^0(\widetilde{C},\omega_{\widetilde{C}})_4=1$.
		The group $\widetilde{G}=\Z_6$ acts on $H^0(\widetilde{C},\omega_{\widetilde{C}})_-$ by $w\mapsto \xi_6 w$ so that we have $H^0(\widetilde{C},\omega_{\widetilde{C}})_{-,i}=H^0(\widetilde{C},\omega_{\widetilde{C}})_i$. 
		Now, we can compute the dimension of $P(\widetilde{G})$: This is equal to $\dim (S^2H^0(\widetilde{C},\omega_{\widetilde{C}})_-)^{\widetilde{G}}$, as we remarked earlier. Note that 
		\begin{equation*}
		(S^2H^0(\widetilde{C},\omega_{\widetilde{C}})_-)^{\widetilde{G}}=H^0(\widetilde{C},\omega_{\widetilde{C}})_1\otimes H^0(\widetilde{C},\omega_{\widetilde{C}})_5
		\end{equation*} 
		So $\dim P(\widetilde{G})=1$. 
		\item As an abelian and non-cyclic example consider the $\mathbb{Z}_{3}\times\mathbb{Z}_{3}$-cover of $\P^1$ given by the matrix 
		$\begin{pmatrix} 1&1&1&0\\
		0&0&2&1
		\end{pmatrix}$. In other words, this family is given by equations
		\begin{gather*}
		w_1^{3}=(z-z_1)(z-z_2)(z-z_3)\\ \nonumber
		w_2^{3}=(z-z_3)^2(z-z_4)
		\end{gather*}
		
		This family is one of families with abelian Galois group which give rise to Special subvarieties in the Torelli locus, see \cite{MO} or \cite{MZ} for more details. \\
		
		Consider the quotient $\mathbb{Z}_{3}\times\mathbb{Z}_{3}\to\mathbb{Z}_{3}$ by the second factor. This corresponds to a cover $\widetilde{C}_t\to C_t$ which is totally ramified in 3 points and $g(C_t)=1$. In fact the quotient curve $C_t$ is just given by the first of the above equations $w^{3}=(z-z_1)(z-z_2)(z-z_3)$ or equivalently by the first row of the above matrix. This family corresponds to the example with data $(r,\tilde g, \#)=(4,4,7)$ in the table. The automorphism $\nu$ of order 3 corresponds to the automorphism $w_1\mapsto w_1, w_2\mapsto \xi_3w_2$, where $\xi_3$ is a primitive 3rd root of unity. Using this action, we compute the eigenspace $H^0(\widetilde{C},\omega_{\widetilde{C}})_-$. We have that
		\[H^0(\widetilde{C},\omega_{\widetilde{C}})_-=H^0(\widetilde{C},\omega_{\widetilde{C}})_{(1,1)}\oplus H^0(\widetilde{C},\omega_{\widetilde{C}})_{(1,2)}\oplus H^0(\widetilde{C},\omega_{\widetilde{C}})_{(2,1)},\]
		where $H^0(\widetilde{C},\omega_{\widetilde{C}})_i$ is the eigenspace w.r.t the character $i\in \mathbb{Z}_{3}\times\mathbb{Z}_{3}$. For an abelian cover, these dimensions are computed in \cite{MZ}, Prop 2.8. We have that $\dim H^0(\widetilde{C},\omega_{\widetilde{C}})_{(1,1)}=\dim H^0(\widetilde{C},\omega_{\widetilde{C}})_{(1,2)}=H^0(\widetilde{C},\omega_{\widetilde{C}})_{(2,1)}=1$. Hence 
		\begin{equation*}
		(S^2H^0(\widetilde{C},\omega_{\widetilde{C}})_-)^{\widetilde{G}}=H^0(\widetilde{C},\omega_{\widetilde{C}})_{(1,2)}\otimes H^0(\widetilde{C},\omega_{\widetilde{C}})_{(2,1)}
		\end{equation*} 
		So $\dim P(\widetilde{G})=1$. 
		
		\item  Consider the family given by the monodromy data $(6,(1,1,1,1,2))$, i.e., the family $y^{6}=(x-t_1)(x-t_2)(x-t_3)(x-t_4)(x-t_5)^2$. This family has Galois group $\mathbb{Z}_{6}$ and fiber genus $7$. The quotient by the subgroup $\Z_3$ gives rise to a triple cover $\widetilde{C}_t\to C_t$ which is totally ramified at 6 points. The quotient curve $C_t$ corresponds to the $\Z_2$-covering $y^{2}=(x-t_1)(x-t_2)(x-t_3)(x-t_4)$, which is a curve of genus 1. Hence the Prym family is contained in $R_{6,[6]}$ and this family corresponds to the example with data $(r,\tilde g, \#)=(5,4,2)$ in the table. \\
		
		The automorphism $\delta$ of order 3 corresponds to the automorphism $y\mapsto \xi_3y$, where $\xi_3$ is a primitive 3rd root of unity. Using this action, we compute the eigenspace $H^0(\widetilde{C},\omega_{\widetilde{C}})_-$. We have that
		\[H^0(\widetilde{C},\omega_{\widetilde{C}})_-=H^0(\widetilde{C},\omega_{\widetilde{C}})_1\oplus H^0(\widetilde{C},\omega_{\widetilde{C}})_2 \oplus H^0(\widetilde{C},\omega_{\widetilde{C}})_4\oplus H^0(\widetilde{C},\omega_{\widetilde{C}})_5,\]
		where $H^0(\widetilde{C},\omega_{\widetilde{C}})_i$ is the eigenspace w.r.t the character $i\in \Z_{6}$. We have the dimensions $\dim H^0(\widetilde{C},\omega_{\widetilde{C}})_1=3$, $\dim H^0(\widetilde{C},\omega_{\widetilde{C}})_2=2$, $\dim H^0(\widetilde{C},\omega_{\widetilde{C}})_4=1$, $\dim H^0(\widetilde{C},\omega_{\widetilde{C}})_5=0$.
		The group $\widetilde{G}=\Z_6$ acts on $H^0(\widetilde{C},\omega_{\widetilde{C}})_-$ by $y\mapsto \xi_6 y$ so that we have $H^0(\widetilde{C},\omega_{\widetilde{C}})_{-,i}=H^0(\widetilde{C},\omega_{\widetilde{C}})_i$. 
		Now, we can compute $\dim P(\widetilde{G})=\dim(S^2H^0(\widetilde{C},\omega_{\widetilde{C}})_-)^{\widetilde{G}}$. Note that 
		\begin{equation*}
		(S^2H^0(\widetilde{C},\omega_{\widetilde{C}})_-)^{\widetilde{G}}=H^0(\widetilde{C},\omega_{\widetilde{C}})_2\otimes H^0(\widetilde{C},\omega_{\widetilde{C}})_4
		\end{equation*} 
		So $\dim P(\widetilde{G})=2$. This implies that the family satisfies condition (A). Since the family is two dimensional, it is not enough to conclude and we must still show that condition (B) holds. In order to do this we use the basis of the differential forms introduced earlier. It holds that
		\begin{gather*}
		H^0(\widetilde{C},\omega_{\widetilde{C}})_{-,2}=H^0(\widetilde{C},\omega_{\widetilde{C}})_2=
		\langle\alpha_1=y^2\prod_{i=1}^{5} (x-t_i)^{-1}dx, \alpha_2=x\alpha_1\rangle
		\end{gather*}
		and
		\begin{gather*}
		H^0(\widetilde{C},\omega_{\widetilde{C}})_{-,4}=
		H^0(\widetilde{C},\omega_{\widetilde{C}})_4=\langle\beta=y^4(x-t_1)^{-1}(x-t_2)^{-1}(x-t_3)^{-1}(x-t_4)^{-1}(x-t_5)^{-2}dx\rangle,
		\end{gather*}
		so that
		$(S^2H^0(\widetilde{C},\omega_{\widetilde{C}})_-)^{\widetilde{G}}=\langle\alpha_1\odot\beta,
		\alpha_2\odot\beta \rangle$. We have
		\[
		m(\alpha_1\odot\beta)=\frac{(dx)^2}{\prod_{i=1}^{5}(x-t_i)},\qquad
		m(\alpha_2\odot\beta)=\frac{x(dx)^2}{\prod_{i=1}^{5}(x-t_i)}.
		\]
		So $v=a_1(\alpha_1\odot\beta)+a_2(\alpha_2\odot\beta)\in\ker(m)$
		if and only if
		$a_1\frac{(dx)^2}{\prod_{i=1}^{5}(x-t_i)}+a_2\frac{x(dx)^2}{\prod_{i=1}^{5}(x-t_i)}=0$.
		It is straightforward to see that this holds if and only if
		$a_1=a_2=0$. This shows that $m$ is injective and by condition (A), it is an isomorphism, so condition (B) is satisfied.
		
		\item Consider the family of genus 2 curves with non-abelian Galois group $\widetilde{G}=D_4$ and ramification data $(2^3,4)$. This family corresponds to the example with data $(r,\tilde g, \#)=(4,2,3)$ in the table and does not satisfy (B1) and therefore we can not conclude by showing the isomorphy of the multiplication map. However, in this case the quotient curve $C$ is isomorphic to $\P^1$ and so by the remark after Theorem \ref{Subtorus-jacobian}, the family of Pryms $P(\widetilde{C}/C)$ is isogeneous to the family of Jacobians. A close inspection of Tables 1,2 in \cite{FGP} shows that this famiy is family (29) of that paper and hence it is a special family. 
		The same argument shows that the families of genus 3 curves with Galois group $\widetilde{G}=D_4$, ramification data $(2^5)$ and $H=\Z_2^2$ which corresponds to examples $(r,\tilde g, \#)=(5,3,1),(5,3,2)$ in the table are isogeneous to the family (32) of \cite{FGP} and so are also special 2-dimensional families (these also do not satisfy (B1)).

		\item
		An abelian example that does not verifies condition is the following family. Consider $\widetilde G=\Z_3^2$ and the monodromy matrix $A=\begin{pmatrix}1&0&1&2&2\\0&2&2&0&2\end{pmatrix}$. Then the curve $\widetilde C$ has genus 7 and equations
		\begin{gather*}
		w_1^3=(z-z_1)(z-z_3)(z-z_4)^2(z-z_5)^2\\
		w_2^3=(z-z_2)^2(z-z_3)^2(z-z_5)^2
		\end{gather*}
		Also consider the subgroup $H\cong \Z_3$ generated by the element $(0,1)^t$, that acts as $w_1\mapsto w_1, w_2\mapsto \xi_3w_2$, where $\xi_3$ is a primitive 3rd root of unity.
		We have
		$H^0(\widetilde{C},\omega_{\widetilde{C}})_-=V_{(0,2)}\oplus V_{(1,1)}\oplus V_{(2,1)}\oplus V_{(1,2)}\oplus V_{(2,2)}$,
		where all summands have dimension 1. Then  we obtain
		\[(S^2H^0(\widetilde{C},\omega_{\widetilde{C}})_-)^{\widetilde G}= (V_{(1,1)}\otimes V_{(2,2)})\oplus (V_{(1,2)}\otimes V_{(2,1)})\]
		hence condition (B1) is not satisfied.
		We have
		\begin{gather*}
		H^0(\widetilde{C},\omega_{\widetilde{C}})_{(1,1)}=\langle\omega_1=\frac{w_1w_2}{\prod_{i=1}^4(z-z_i)(z-z_5)^2}dx\rangle\\
		H^0(\widetilde{C},\omega_{\widetilde{C}})_{(2,2)}=\langle\omega_2=\frac{w_1^2w_2^2}{(z-z_1)\prod_{i=2}^4(z-z_i)^2(z-z_5)^3}dx\rangle\\
		H^0(\widetilde{C},\omega_{\widetilde{C}})_{(1,2)}=\langle\omega_3=\frac{w_1w_2^2}{(z-z_1)(z-z_2)^2(z-z_3)^2(z-z_4)(z-z_5)^2}dx\rangle\\
		H^0(\widetilde{C},\omega_{\widetilde{C}})_{(2,1)}=\langle\omega_4=\frac{w_1^2w_2}{(z-z_1)(z-z_2)(z-z_3)^2(z-z_4)^2(z-z_5)^2}dx\rangle
		\end{gather*}
		Hence we compute
		\begin{gather*}
		v_1:= m(\omega_1\odot\omega_2)=\frac{(dx)^2}{(z-z_1)(z-z_2)(z-z_4)(z-z_5)}\\
		v_2:= m(\omega_3\odot\omega_4)=\frac{(dx)^2}{(z-z_1)(z-z_2)(z-z_3)(z-z_4)}
		\end{gather*}
		and we find that $a_1v_1+a_2v_2=0$ if and only if $a_1=a_2=0$, so $m$ is injective. Together with condition (A) this implies that (B) holds.
		Thus the family gives rise to a 2-dimensional Shimura variety.
		
		\item
		Consider the 3-dimensional family with group $\widetilde G=\Z_2^3$ and monodromy matrix $A=\begin{pmatrix}0&0&1&1&0&0\\0&1&1&1&0&1\\1&1&1&1&1&1\end{pmatrix}$.
		The equations of the genus 5 curve $\widetilde C$ are
		\begin{gather*}
		w_1^2=(z-z_3)(z-z_4)\\
		w_2^2=(z-z_2)(z-z_3)(z-z_4)(z-z_6)\\
		w_3^2=(z-z_1)(z-z_2)(z-z_3)(z-z_4)(z-z_5)(z-z_6)
		\end{gather*}
		We consider the $\Z_2^2$-cover given by the action of the subgroup $H=\langle(1,0,0)^t,(0,1,0)^t\rangle$. The quotient curve $C=\widetilde C/H$ has genus 2.
		We have 
		$H^0(\widetilde{C},\omega_{\widetilde{C}})_-=H^0(\widetilde{C},\omega_{\widetilde{C}})_{(0,1,0)}\oplus H^0(\widetilde{C},\omega_{\widetilde{C}})_{(1,0,1)}\oplus H^0(\widetilde{C},\omega_{\widetilde{C}})_{(1,1,1)}$ where each summand has dimension 1.
		Then we get
		\begin{gather*}
		H^0(\widetilde{C},\omega_{\widetilde{C}})_{(0,1,0)}=\langle{\omega_1=\frac{y_2}{(z-z_2)(z-z_3)(z-z_4)(z-z_6)}dx}\rangle\\
		H^0(\widetilde{C},\omega_{\widetilde{C}})_{(1,0,1)}=\langle{\omega_2=\frac{y_1y_3}{(z-z_1)(z-z_2)(z-z_3)(z-z_4)(z-z_5)(z-z_6)}dx}\rangle\\
		H^0(\widetilde{C},\omega_{\widetilde{C}})_{(1,1,1)}=\langle{\omega_3=\frac{y_1y_2y_3}{(z-z_1)(z-z_2)(z-z_3)^2(z-z_4)^2(z-z_5)(z-z_6)}dx}\rangle
		\end{gather*}
		thus if we set
		\begin{gather*}
		v_1:= m(\omega_1\odot\omega_1)=\frac{(dx)^2}{(z-z_2)(z-z_3)(z-z_4)(z-z_6)}\\
		v_2:= m(\omega_2\odot\omega_2)=\frac{(dx)^2}{(z-z_1)(z-z_2)(z-z_5)(z-z_6)}\\
		v_3:= m(\omega_3\odot\omega_3)=\frac{(dx)^2}{(z-z_1)(z-z_3)(z-z_4)(z-z_5)}
		\end{gather*}
		we find that $a_1v_1+a_2v_2+a_3v_3=0$ if and only if $a_1(z-z_1)(z-z_5)+a_2(z-z_3)(z-z_4)+a_3(z-z_2)(z-z_6)=0$, i.e.\ $a_1=a_2=a_3=0$.
		Hence the multiplication map is an isomorphism and the family gives rise to a special subvariety.

	\end{itemize}
	
	\newpage
	
	We list all the obtained prym data that give rise to shimura varieties.
	For each example it is reported: the number $r$ of critical values on $\P^1$, the genus $\widetilde g$ of $\widetilde C$ and $g$ of $C$, the dimension $p=\widetilde g-g$ and progressive index ($\#$), the group $\widetilde G$ and the subgroup $H$ determining the prym cover, the number of ramification and branch points of this cover, the quotient group $G=\widetilde G/H$ acting on $C$. Finally the fulfilled conditions are marked.
	
	
	\def\headtable{\begin{table}[H]
			\begin{tabular}{*{13}{c}}
				\toprule
				$r$&$\widetilde{g}$&$g$&$p$&$\#$&$\widetilde{G}$&$H$&Ram pt&Br pt&$G$&(B1)&(B2)&(B) \\ 
				\midrule}
			\def\tailtable{\bottomrule\end{tabular}\end{table}}
	\def\breaktable{\tailtable \newpage \headtable}
	
	\headtable
	4 & 2 & 0 & 2 & 1 &$S_3$&$C_3$& 4 & 4 &$C_2$& \phantom{X} & $\checkmark$ & $\checkmark$ \\ 
	4 & 2 & 0 & 2 & 2 &$C_6$&$C_3$& 4 & 4 &$C_2$& $\checkmark$ & $\checkmark$ & $\checkmark$ \\ 
	4 & 2 & 0 & 2 & 3 &$D_4$&$C_4$& 6 & 4 &$C_2$& \phantom{X} & $\checkmark$ & $\checkmark$ \\ 
	4 & 2 & 0 & 2 & 4,5 &$D_4$&$C_2^2$& 10 & 5 &$C_2$& \phantom{X} & $\checkmark$ & $\checkmark$ \\ 
	4 & 2 & 0 & 2 & 6 &$D_6$&$C_3$& 4 & 4 &$C_2^2$& \phantom{X} & $\checkmark$ & $\checkmark$ \\ 
	4 & 2 & 0 & 2 & 7 &$D_6$&$C_6$& 10 & 4 &$C_2$& \phantom{X} & $\checkmark$ & $\checkmark$ \\ 
	4 & 2 & 0 & 2 & 8,9 &$D_6$&$S_3$& 10 & 4 &$C_2$& \phantom{X} & $\checkmark$ & $\checkmark$ \\ 
	4 & 3 & 1 & 2 & 1 &$C_6$&$C_3$& 2 & 2 &$C_2$& $\checkmark$ & \phantom{X} & $\checkmark$ \\ 
	4 & 3 & 0 & 3 & 2 &$C_6$&$C_3$& 5 & 5 &$C_2$& $\checkmark$ & $\checkmark$ & $\checkmark$ \\ 
	4 & 3 & 0 & 3 & 3 &$C_2\times C_4$&$C_4$& 4 & 4 &$C_2$& $\checkmark$ & $\checkmark$ & $\checkmark$ \\ 
	4 & 3 & 0 & 3 & 4 &$C_2\times C_4$&$C_2^2$& 12 & 6 &$C_2$& $\checkmark$ & $\checkmark$ & $\checkmark$ \\ 
	4 & 3 & 1 & 2 & 5 &$C_2\times C_4$&$C_4$& 4 & 2 &$C_2$& $\checkmark$ & $\checkmark$ & $\checkmark$ \\ 
	4 & 3 & 0 & 3 & 6,7 &$C_2\times C_4$&$C_4$& 8 & 5 &$C_2$& $\checkmark$ & $\checkmark$ & $\checkmark$ \\ 
	4 & 3 & 0 & 3 & 8 &$C_2\times C_4$&$C_2^2$& 12 & 6 &$C_2$& $\checkmark$ & $\checkmark$ & $\checkmark$ \\ 
	4 & 3 & 0 & 3 & 10 &$A_4$&$C_2^2$& 12 & 6 &$C_3$& \phantom{X} & $\checkmark$ & $\checkmark$ \\ 
	4 & 3 & 1 & 2 & 13-15 &$C_2\times D_4$&$C_4$& 4 & 2 &$C_2^2$& \phantom{X} & $\checkmark$ & $\checkmark$ \\ 
	4 & 3 & 1 & 2 & 16 &$C_2\times D_4$&$D_4$& 4 & 1 &$C_2$& \phantom{X} & $\checkmark$ & $\checkmark$ \\ 
	4 & 3 & 0 & 3 & 17-19 &$D_4\rtimes C_2$&$C_2^2$& 12 & 6 &$C_2^2$& $\checkmark$ & $\checkmark$ & $\checkmark$ \\ 
	4 & 3 & 0 & 3 & 20 &$D_4\rtimes C_2$&$C_4$& 4 & 4 &$C_2^2$& $\checkmark$ & $\checkmark$ & $\checkmark$ \\ 
	4 & 3 & 0 & 3 & 21-23 &$D_4\rtimes C_2$&$D_4$& 20 & 5 &$C_2$& $\checkmark$ & $\checkmark$ & $\checkmark$ \\ 
	4 & 3 & 0 & 3 & 24-26 &$D_4\rtimes C_2$&$C_2\times C_4$& 12 & 4 &$C_2$& $\checkmark$ & $\checkmark$ & $\checkmark$ \\ 
	4 & 3 & 0 & 3 & 27 &$S_4$&$C_2^2$& 12 & 6 &$S_3$& \phantom{X} & $\checkmark$ & $\checkmark$ \\ 
	4 & 3 & 0 & 3 & 28 &$S_4$&$A_4$& 20 & 4 &$C_2$& \phantom{X} & $\checkmark$ & $\checkmark$ \\ 
	4 & 4 & 0 & 4 & 1 &$C_6$&$C_3$& 6 & 6 &$C_2$& $\checkmark$ & $\checkmark$ & $\checkmark$ \\ 
	4 & 4 & 0 & 4 & 2-4 &$Q_8$&$C_4$& 10 & 6 &$C_2$& \phantom{X} & $\checkmark$ & $\checkmark$ \\ 
	4 & 4 & 2 & 2 & 5,6 &$C_3^2$&$C_3$& 0 & 0 &$C_3$& $\checkmark$ & $\checkmark$ & $\checkmark$ \\ 
	4 & 4 & 1 & 3 & 7,8 &$C_3^2$&$C_3$& 3 & 3 &$C_3$& $\checkmark$ & $\checkmark$ & $\checkmark$ \\ 
	4 & 4 & 0 & 4 & 9 &$C_3^2$&$C_3$& 6 & 6 &$C_3$& $\checkmark$ & $\checkmark$ & $\checkmark$ \\ 
	4 & 4 & 0 & 4 & 11 &$C_2\times C_6$&$C_3$& 6 & 6 &$C_2^2$& $\checkmark$ & $\checkmark$ & $\checkmark$ \\ 
	4 & 4 & 0 & 4 & 12 &$C_2\times C_6$&$C_2^2$& 14 & 7 &$C_3$& $\checkmark$ & $\checkmark$ & $\checkmark$ \\ 
	4 & 4 & 0 & 4 & 13,15 &$C_2\times C_6$&$C_6$& 12 & 5 &$C_2$& $\checkmark$ & $\checkmark$ & $\checkmark$ \\ 
	4 & 4 & 0 & 4 & 14 &$C_2\times C_6$&$C_6$& 6 & 4 &$C_2$& $\checkmark$ & $\checkmark$ & $\checkmark$ \\ 
	4 & 4 & 2 & 2 & 16 &$C_3\times S_3$&$C_3$& 0 & 0 &$S_3$& $\checkmark$ & $\checkmark$ & $\checkmark$ \\ 
	4 & 4 & 2 & 2 & 17 &$C_3\times S_3$&$C_3$& 0 & 0 &$C_6$& \phantom{X} & $\checkmark$ & $\checkmark$ \\ 
	4 & 4 & 0 & 4 & 18 &$C_3\times S_3$&$C_3$& 6 & 6 &$S_3$& $\checkmark$ & $\checkmark$ & $\checkmark$ \\ 
	4 & 4 & 0 & 4 & 19 &$C_3\times S_3$&$S_3$& 18 & 6 &$C_3$& $\checkmark$ & $\checkmark$ & $\checkmark$ \\ 
	4 & 4 & 0 & 4 & 20 &$C_3\times S_3$&$C_3^2$& 12 & 4 &$C_2$& $\checkmark$ & $\checkmark$ & $\checkmark$ \\ 
	4 & 4 & 2 & 2 & 21,22 &$C_3\rtimes S_3$&$C_3$& 0 & 0 &$S_3$& \phantom{X} & $\checkmark$ & $\checkmark$ \\ 
	4 & 4 & 2 & 2 & 25,26 &$S_3^2$&$C_3$& 0 & 0 &$D_6$& \phantom{X} & $\checkmark$ & $\checkmark$ \\ 
	
	\breaktable
	
	4 & 5 & 0 & 5 & 1 &$C_8$&$C_4$& 8 & 6 &$C_2$& $\checkmark$ & $\checkmark$ & $\checkmark$ \\ 
	4 & 5 & 1 & 4 & 2 &$C_2\times C_4$&$C_2^2$& 8 & 4 &$C_2$& $\checkmark$ & \phantom{X} & $\checkmark$ \\ 
	4 & 5 & 0 & 5 & 5 &$C_3\rtimes C_4$&$C_6$& 16 & 6 &$C_2$& \phantom{X} & $\checkmark$ & $\checkmark$ \\ 
	4 & 5 & 1 & 4 & 8 &$C_2\times C_6$&$C_3$& 4 & 4 &$C_2^2$& $\checkmark$ & \phantom{X} & $\checkmark$ \\ 
	4 & 5 & 1 & 4 & 9 &$C_2\times C_6$&$C_6$& 4 & 2 &$C_2$& $\checkmark$ & \phantom{X} & $\checkmark$ \\ 
	4 & 5 & 2 & 3 & 12 &$C_2^2\rtimes C_4$&$C_2^2$& 0 & 0 &$C_4$& \phantom{X} & $\checkmark$ & $\checkmark$ \\ 
	4 & 5 & 2 & 3 & 15 &$C_2^2\times C_4$&$C_2^2$& 0 & 0 &$C_4$& $\checkmark$ & $\checkmark$ & $\checkmark$ \\ 
	4 & 5 & 1 & 4 & 16,18 &$C_2^2\times C_4$&$C_4$& 8 & 4 &$C_2^2$& $\checkmark$ & $\checkmark$ & $\checkmark$ \\ 
	4 & 5 & 1 & 4 & 17 &$C_2^2\times C_4$&$C_2^2$& 8 & 4 &$C_2^2$& $\checkmark$ & $\checkmark$ & $\checkmark$ \\ 
	4 & 5 & 1 & 4 & 19 &$C_2^2\times C_4$&$C_2\times C_4$& 8 & 2 &$C_2$& $\checkmark$ & $\checkmark$ & $\checkmark$ \\ 
	4 & 5 & 2 & 3 & 29 &$C_2\times A_4$&$C_2^2$& 0 & 0 &$C_6$& \phantom{X} & $\checkmark$ & $\checkmark$ \\ 
	4 & 5 & 2 & 3 & 30 &$C_2^2\rtimes D_4$&$C_2^2$& 0 & 0 &$D_4$& $\checkmark$ & $\checkmark$ & $\checkmark$ \\ 
	4 & 5 & 1 & 4 & 31 &$C_2^2\rtimes D_4$&$C_2^2$& 8 & 4 &$C_2^3$& \phantom{X} & $\checkmark$ & $\checkmark$ \\ 
	4 & 5 & 1 & 4 & 32,33 &$C_2^2\rtimes D_4$&$C_4$& 8 & 4 &$D_4$& \phantom{X} & $\checkmark$ & $\checkmark$ \\ 
	4 & 5 & 1 & 4 & 34-36 &$C_2^2\rtimes D_4$&$C_2\times C_4$& 8 & 2 &$C_2^2$& \phantom{X} & $\checkmark$ & $\checkmark$ \\ 
	4 & 5 & 1 & 4 & 37 &$C_2^2\rtimes D_4$&$C_4\rtimes C_4$& 8 & 1 &$C_2$& \phantom{X} & $\checkmark$ & $\checkmark$ \\ 
	4 & 5 & 2 & 3 & 45 &$C_2\times S_4$&$C_2^2$& 0 & 0 &$D_6$& \phantom{X} & $\checkmark$ & $\checkmark$ \\ 
	4 & 6 & 0 & 6 & 1 &$C_{10}$&$C_5$& 5 & 5 &$C_2$& $\checkmark$ & $\checkmark$ & $\checkmark$ \\ 
	4 & 7 & 1 & 6 & 1,2 &$C_8$&$C_4$& 4 & 4 &$C_2$& $\checkmark$ & \phantom{X} & $\checkmark$ \\ 
	4 & 7 & 1 & 6 & 3 &$C_9$&$C_3$& 6 & 6 &$C_3$& $\checkmark$ & \phantom{X} & $\checkmark$ \\ 
	4 & 7 & 1 & 6 & 4 &$C_{10}$&$C_5$& 3 & 3 &$C_2$& $\checkmark$ & \phantom{X} & $\checkmark$ \\ 
	4 & 7 & 1 & 6 & 8 &$C_{12}$&$C_3$& 6 & 6 &$C_4$& $\checkmark$ & \phantom{X} & $\checkmark$ \\ 
	4 & 7 & 1 & 6 & 9 &$C_{12}$&$C_4$& 8 & 5 &$C_3$& $\checkmark$ & \phantom{X} & $\checkmark$ \\ 
	4 & 7 & 0 & 7 & 10 &$C_{12}$&$C_6$& 12 & 6 &$C_2$& $\checkmark$ & $\checkmark$ & $\checkmark$ \\ 
	4 & 7 & 2 & 5 & 11 &$C_{12}$&$C_3$& 3 & 3 &$C_4$& $\checkmark$ & \phantom{X} & $\checkmark$ \\ 
	4 & 7 & 1 & 6 & 12 &$C_2\times C_6$&$C_3$& 6 & 6 &$C_2^2$& $\checkmark$ & $\checkmark$ & $\checkmark$ \\ 
	4 & 7 & 1 & 6 & 13 &$C_2\times C_6$&$C_6$& 6 & 3 &$C_2$& $\checkmark$ & $\checkmark$ & $\checkmark$ \\ 
	4 & 7 & 2 & 5 & 14,15 &$C_4^2$&$C_4$& 4 & 2 &$C_4$& $\checkmark$ & $\checkmark$ & $\checkmark$ \\ 
	4 & 7 & 2 & 5 & 16,17 &$C_4\rtimes C_4$&$C_4$& 4 & 2 &$C_4$& \phantom{X} & $\checkmark$ & $\checkmark$ \\ 
	4 & 7 & 1 & 6 & 18 &$C_2\times C_8$&$C_4$& 4 & 4 &$C_2^2$& $\checkmark$ & $\checkmark$ & $\checkmark$ \\ 
	4 & 7 & 2 & 5 & 19 &$C_2\times C_8$&$C_4$& 4 & 2 &$C_4$& $\checkmark$ & $\checkmark$ & $\checkmark$ \\ 
	4 & 7 & 1 & 6 & 20 &$C_2\times C_8$&$C_8$& 4 & 2 &$C_2$& $\checkmark$ & $\checkmark$ & $\checkmark$ \\ 
	4 & 7 & 1 & 6 & 23-25 &$C_2\times Q_8$&$C_4$& 12 & 6 &$C_2^2$& \phantom{X} & $\checkmark$ & $\checkmark$ \\ 
	4 & 7 & 1 & 6 & 26 &$C_2\times Q_8$&$Q_8$& 12 & 3 &$C_2$& \phantom{X} & $\checkmark$ & $\checkmark$ \\ 
	4 & 7 & 2 & 5 & 29 &$C_3\times S_3$&$C_3$& 3 & 3 &$S_3$& $\checkmark$ & $\checkmark$ & $\checkmark$ \\ 
	4 & 7 & 3 & 4 & 30 &$C_3\times S_3$&$C_3$& 0 & 0 &$C_6$& \phantom{X} & $\checkmark$ & $\checkmark$ \\ 
	4 & 7 & 2 & 5 & 31 &$C_3\times S_3$&$C_3$& 3 & 3 &$S_3$& \phantom{X} & \phantom{X} & \phantom{X} \\ 
	4 & 7 & 2 & 5 & 32 &$C_3\times C_6$&$C_3$& 3 & 3 &$C_6$& $\checkmark$ & $\checkmark$ & $\checkmark$ \\ 
	4 & 7 & 3 & 4 & 33 &$C_3\times C_6$&$C_3$& 0 & 0 &$C_6$& $\checkmark$ & $\checkmark$ & $\checkmark$ \\ 
	4 & 7 & 1 & 6 & 41 &$D_8\rtimes C_2$&$C_4$& 4 & 4 &$C_2^3$& \phantom{X} & $\checkmark$ & $\checkmark$ \\ 
	4 & 7 & 2 & 5 & 42 &$D_8\rtimes C_2$&$C_4$& 4 & 2 &$D_4$& $\checkmark$ & $\checkmark$ & $\checkmark$ \\ 
	4 & 7 & 1 & 6 & 43-45 &$D_8\rtimes C_2$&$C_8$& 4 & 2 &$C_2^2$& \phantom{X} & $\checkmark$ & $\checkmark$ \\ 
	4 & 7 & 1 & 6 & 46 &$D_8\rtimes C_2$&$Q_{16}$& 4 & 1 &$C_2$& \phantom{X} & $\checkmark$ & $\checkmark$ \\ 
	
	\breaktable
	
	4 & 8 & 2 & 6 & 1 &$C_9$&$C_3$& 4 & 4 &$C_3$& $\checkmark$ & \phantom{X} & $\checkmark$ \\ 
	4 & 9 & 1 & 8 & 1 &$C_{10}$&$C_5$& 4 & 4 &$C_2$& $\checkmark$ & $\checkmark$ & $\checkmark$ \\ 
	4 & 9 & 2 & 7 & 2 &$C_{12}$&$C_4$& 4 & 3 &$C_3$& $\checkmark$ & \phantom{X} & $\checkmark$ \\ 
	4 & 9 & 3 & 6 & 3 &$C_{12}$&$C_3$& 2 & 2 &$C_4$& $\checkmark$ & \phantom{X} & $\checkmark$ \\ 
	4 & 9 & 3 & 6 & 4,5 &$C_4^2$&$C_4$& 0 & 0 &$C_4$& $\checkmark$ & \phantom{X} & $\checkmark$ \\ 
	4 & 9 & 3 & 6 & 6 &$C_4^2$&$C_4$& 0 & 0 &$C_4$& $\checkmark$ & $\checkmark$ & $\checkmark$ \\ 
	4 & 9 & 1 & 8 & 8 &$C_2\times C_8$&$C_4$& 8 & 6 &$C_2^2$& $\checkmark$ & $\checkmark$ & $\checkmark$ \\ 
	4 & 9 & 1 & 8 & 9 &$C_2\times C_8$&$C_8$& 8 & 3 &$C_2$& $\checkmark$ & $\checkmark$ & $\checkmark$ \\ 
	4 & 9 & 2 & 7 & 10 &$C_2\times C_8$&$C_4$& 8 & 4 &$C_4$& $\checkmark$ & $\checkmark$ & $\checkmark$ \\ 
	4 & 9 & 3 & 6 & 13 &$C_2^2\times C_6$&$C_2^2$& 0 & 0 &$C_6$& $\checkmark$ & $\checkmark$ & $\checkmark$ \\ 
	4 & 9 & 3 & 6 & 16 &$C_4\wr C_2$&$C_4$& 0 & 0 &$D_4$& $\checkmark$ & $\checkmark$ & $\checkmark$ \\ 
	4 & 9 & 3 & 6 & 18 &$C_4\times D_4$&$C_4$& 0 & 0 &$C_2\times C_4$& \phantom{X} & $\checkmark$ & $\checkmark$ \\ 
	4 & 9 & 3 & 6 & 22 &$C_2\times C_3\rtimes D_4$&$C_2^2$& 0 & 0 &$D_6$& \phantom{X} & $\checkmark$ & $\checkmark$ \\ 
	4 & 9 & 3 & 6 & 23,24 &$D_4.D_4$&$C_4$& 0 & 0 &$C_2\times D_4$& \phantom{X} & $\checkmark$ & $\checkmark$ \\ 
	4 & 10 & 3 & 7 & 1,3 &$C_{12}$&$C_3$& 3 & 3 &$C_4$& $\checkmark$ & \phantom{X} & $\checkmark$ \\ 
	4 & 10 & 1 & 9 & 2 &$C_{12}$&$C_6$& 6 & 4 &$C_2$& $\checkmark$ & \phantom{X} & $\checkmark$ \\ 
	4 & 10 & 2 & 8 & 4 &$C_{12}$&$C_3$& 6 & 6 &$C_4$& $\checkmark$ & $\checkmark$ & $\checkmark$ \\ 
	4 & 10 & 2 & 8 & 5 &$C_{12}$&$C_4$& 6 & 4 &$C_3$& $\checkmark$ & $\checkmark$ & $\checkmark$ \\ 
	4 & 10 & 1 & 9 & 6,7 &$C_{14}$&$C_7$& 3 & 3 &$C_2$& $\checkmark$ & \phantom{X} & $\checkmark$ \\ 
	4 & 10 & 4 & 6 & 8 &$C_3\times C_6$&$C_3$& 0 & 0 &$C_6$& $\checkmark$ & $\checkmark$ & $\checkmark$ \\ 
	4 & 10 & 2 & 8 & 9 &$C_3\times C_6$&$C_3$& 6 & 6 &$C_6$& $\checkmark$ & $\checkmark$ & $\checkmark$ \\ 
	4 & 10 & 2 & 8 & 10,11 &$C_3\times C_6$&$C_6$& 6 & 2 &$C_3$& $\checkmark$ & $\checkmark$ & $\checkmark$ \\ 
	4 & 10 & 4 & 6 & 12 &$C_3\times C_6$&$C_3$& 0 & 0 &$C_6$& $\checkmark$ & $\checkmark$ & $\checkmark$ \\ 
	4 & 10 & 2 & 8 & 13 &$C_3\times D_4$&$C_3$& 6 & 6 &$D_4$& $\checkmark$ & $\checkmark$ & $\checkmark$ \\ 
	4 & 10 & 2 & 8 & 14 &$C_3\times D_4$&$C_4$& 6 & 4 &$C_6$& \phantom{X} & $\checkmark$ & $\checkmark$ \\ 
	4 & 10 & 4 & 6 & 18-20 &$C_3^3$&$C_3$& 0 & 0 &$C_3^2$& $\checkmark$ & $\checkmark$ & $\checkmark$ \\ 
	4 & 10 & 2 & 8 & 25 &$C_6\times S_3$&$C_3$& 6 & 6 &$D_6$& $\checkmark$ & $\checkmark$ & $\checkmark$ \\ 
	4 & 10 & 4 & 6 & 26 &$C_6\times S_3$&$C_3$& 0 & 0 &$C_2\times C_6$& \phantom{X} & $\checkmark$ & $\checkmark$ \\ 
	4 & 10 & 2 & 8 & 27 &$C_6\times S_3$&$C_6$& 6 & 2 &$C_6$& \phantom{X} & $\checkmark$ & $\checkmark$ \\ 
	4 & 10 & 4 & 6 & 28 &$C_6\times S_3$&$C_3$& 0 & 0 &$D_6$& $\checkmark$ & $\checkmark$ & $\checkmark$ \\ 
	4 & 10 & 4 & 6 & 33 &$C_3\times C_3\rtimes S_3$&$C_3$& 0 & 0 &$C_3\rtimes S_3$& $\checkmark$ & $\checkmark$ & $\checkmark$ \\ 
	4 & 10 & 4 & 6 & 34,35 &$C_3\times C_3\rtimes S_3$&$C_3$& 0 & 0 &$C_3\times S_3$& \phantom{X} & $\checkmark$ & $\checkmark$ \\ 
	4 & 10 & 4 & 6 & 39-41 &$C_3^2\rtimes C_2^2$&$C_3$& 0 & 0 &$S_3^2$& \phantom{X} & $\checkmark$ & $\checkmark$ \\ 
	4 & 11 & 3 & 8 & 1 &$C_2\times C_{12}$&$C_4$& 4 & 2 &$C_6$& $\checkmark$ & $\checkmark$ & $\checkmark$ \\ 
	4 & 11 & 3 & 8 & 2 &$D_{12}\rtimes C_3$&$C_4$& 4 & 2 &$D_6$& \phantom{X} & $\checkmark$ & $\checkmark$ \\ 
	4 & 12 & 2 & 10 & 1 &$C_{15}$&$C_5$& 3 & 3 &$C_3$& $\checkmark$ & \phantom{X} & $\checkmark$ \\ 
	4 & 12 & 3 & 9 & 2 &$C_{18}$&$C_3$& 5 & 5 &$C_6$& $\checkmark$ & \phantom{X} & $\checkmark$ \\ 
	4 & 13 & 4 & 9 & 1,2 &$C_{15}$&$C_3$& 3 & 3 &$C_5$& $\checkmark$ & \phantom{X} & $\checkmark$ \\ 
	4 & 13 & 5 & 8 & 3 &$C_3\times C_6$&$C_3$& 0 & 0 &$C_6$& $\checkmark$ & \phantom{X} & $\checkmark$ \\ 
	4 & 14 & 4 & 10 & 1 &$C_{15}$&$C_3$& 4 & 4 &$C_5$& $\checkmark$ & $\checkmark$ & $\checkmark$ \\ 
	4 & 14 & 4 & 10 & 2 &$C_3\times D_5$&$C_3$& 4 & 4 &$D_5$& $\checkmark$ & $\checkmark$ & $\checkmark$ \\ 
	4 & 16 & 6 & 10 & 1 &$C_3\times C_9$&$C_3$& 0 & 0 &$C_9$& $\checkmark$ & $\checkmark$ & $\checkmark$ \\ 
	4 & 16 & 6 & 10 & 4 &$C_3\times D_9$&$C_3$& 0 & 0 &$D_9$& $\checkmark$ & $\checkmark$ & $\checkmark$ \\ 
	
	\breaktable
	
	5 & 3 & 0 & 3 & 1,2 &$D_4$&$C_2^2$& 12 & 6 &$C_2$& \phantom{X} & $\checkmark$ & $\checkmark$ \\ 
	5 & 3 & 1 & 2 & 3-5 &$C_2^3$&$C_2^2$& 4 & 2 &$C_2$& \phantom{X} & $\checkmark$ & $\checkmark$ \\ 
	5 & 4 & 1 & 3 & 1 &$C_6$&$C_3$& 3 & 3 &$C_2$& $\checkmark$ & \phantom{X} & $\checkmark$ \\ 
	5 & 4 & 0 & 4 & 2 &$C_6$&$C_3$& 6 & 6 &$C_2$& $\checkmark$ & $\checkmark$ & $\checkmark$ \\ 
	5 & 5 & 1 & 4 & 1 &$C_6$&$C_3$& 4 & 4 &$C_2$& $\checkmark$ & \phantom{X} & $\checkmark$ \\ 
	5 & 5 & 1 & 4 & 2 &$C_2\times C_4$&$C_4$& 8 & 4 &$C_2$& \phantom{X} & \phantom{X} & $\checkmark$ \\ 
	5 & 5 & 1 & 4 & 3,4 &$C_2\times C_4$&$C_4$& 4 & 3 &$C_2$& \phantom{X} & \phantom{X} & $\checkmark$ \\ 
	5 & 5 & 1 & 4 & 5 &$C_2\times C_4$&$C_4$& 8 & 4 &$C_2$& $\checkmark$ & $\checkmark$ & $\checkmark$ \\ 
	5 & 5 & 2 & 3 & 17 &$C_2\times D_4$&$C_2^2$& 0 & 0 &$C_2^2$& \phantom{X} & $\checkmark$ & $\checkmark$ \\ 
	5 & 7 & 1 & 6 & 1 &$C_6$&$C_3$& 6 & 6 &$C_2$& $\checkmark$ & $\checkmark$ & $\checkmark$ \\ 
	5 & 7 & 3 & 4 & 2,3 &$C_3^2$&$C_3$& 0 & 0 &$C_3$& $\checkmark$ & $\checkmark$ & $\checkmark$ \\ 
	5 & 7 & 2 & 5 & 4 &$C_3^2$&$C_3$& 3 & 3 &$C_3$& $\checkmark$ & $\checkmark$ & $\checkmark$ \\ 
	5 & 7 & 2 & 5 & 5-7 &$C_3^2$&$C_3$& 3 & 3 &$C_3$& \phantom{X} & \phantom{X} & $\checkmark$ \\ 
	5 & 7 & 3 & 4 & 8 &$C_3^2$&$C_3$& 0 & 0 &$C_3$& \phantom{X} & \phantom{X} & $\checkmark$ \\ 
	5 & 9 & 1 & 8 & 1 &$C_8$&$C_4$& 8 & 6 &$C_2$& $\checkmark$ & \phantom{X} & $\checkmark$ \\ 
	5 & 9 & 3 & 6 & 3-12 &$C_4.C_2^3$&$C_4$& 0 & 0 &$C_2^3$& \phantom{X} & $\checkmark$ & $\checkmark$ \\ 
	5 & 10 & 2 & 8 & 1 &$C_2\times C_6$&$C_3$& 6 & 6 &$C_2^2$& $\checkmark$ & $\checkmark$ & $\checkmark$ \\ 
	5 & 10 & 4 & 6 & 5 &$C_3\times S_3$&$C_3$& 0 & 0 &$S_3$& $\checkmark$ & $\checkmark$ & $\checkmark$ \\ 
	5 & 12 & 3 & 9 & 1 &$C_9$&$C_3$& 5 & 5 &$C_3$& $\checkmark$ & \phantom{X} & $\checkmark$ \\ 
	5 & 13 & 4 & 9 & 1 &$C_{12}$&$C_3$& 3 & 3 &$C_4$& $\checkmark$ & \phantom{X} & $\checkmark$ \\ 
	6 & 5 & 2 & 3 & 3 &$C_2^3$&$C_2^2$& 0 & 0 &$C_2$& \phantom{X} & $\checkmark$ & $\checkmark$ \\ 
	6 & 5 & 2 & 3 & 4 &$C_2^3$&$C_2^2$& 0 & 0 &$C_2$& \phantom{X} & \phantom{X} & $\checkmark$ \\ 
	6 & 7 & 2 & 5 & 1 &$C_6$&$C_3$& 3 & 3 &$C_2$& $\checkmark$ & \phantom{X} & $\checkmark$ \\ 
	6 & 7 & 2 & 5 & 2,3 &$C_2\times C_4$&$C_4$& 4 & 2 &$C_2$& \phantom{X} & \phantom{X} & $\checkmark$ \\ 
	6 & 10 & 2 & 8 & 1 &$C_6$&$C_3$& 6 & 6 &$C_2$& $\checkmark$ & $\checkmark$ & $\checkmark$ \\ 
	6 & 10 & 4 & 6 & 2 &$C_3^2$&$C_3$& 0 & 0 &$C_3$& $\checkmark$ & $\checkmark$ & $\checkmark$ \\ 
	6 & 10 & 4 & 6 & 3-5 &$C_3^2$&$C_3$& 0 & 0 &$C_3$& \phantom{X} & \phantom{X} & $\checkmark$ \\ 

	\tailtable


\end{document}